\documentclass[reqno]{amsart}

\usepackage{hyperref}
\usepackage{mathtools}
\usepackage{enumerate}
\usepackage[font=footnotesize]{caption}

\title[Convergence for the mean-field dispersion process]{Quantitative convergence guarantees for the mean-field dispersion process}

\author{Fei Cao}
\address{Department of Mathematics and Statistics, University of Massachusetts Amherst, 710 N Pleasant St, Amherst, MA 01003, USA}
\email{fcao@umass.edu}

\author{Jincheng Yang}
\address{Department of Applied Mathematics and Statistics, Johns Hopkins University, 3400 N Charles St, Baltimore, MD 21218, USA}
\email{jincheng@jhu.edu}

\date{\today}

\thanks{\textit{Acknowledgments}. It is a great pleasure to express our gratitude to Sebastien Motsch for generating Figure \ref{fig:illustration_model} that illustrates the dispersion model under investigation. The second author would like to thank Yangyang Li and Minh Pham for insightful discussion on integral equations. The second author was partially supported by the US National Science Foundation under Grant No.~DMS-1926686.}

\newtheorem{theorem}{Theorem}
\newtheorem{proposition}{Proposition}[section]
\newtheorem{corollary}[proposition]{Corollary}

\newtheorem{lemma}[proposition]{Lemma}
\newtheorem{remark}[proposition]{Remark}

\DeclareMathOperator{\expo}{e}
\DeclareMathOperator*{\argmin}{\arg\!\min}
\newcommand{\overbar}[1]{\mkern 1.5mu\overline{\mkern-1.5mu#1\mkern-1.5mu}\mkern 1.5mu}
\newcommand{\dd}{\mathop{\kern0pt\mathrm{d}}\!{}}
\newcommand{\bp}{{\bf p}}
\newcommand{\bq}{{\bf q}}
\newcommand{\pb}{\overbar p}
\newcommand{\bpb}{\overbar \bp}
\newcommand{\cE}{\mathcal E}
\newcommand{\cO}{\mathcal O}
\newcommand{\cS}{\mathcal S}

\begin{document}

\begin{abstract}
We study the discrete Fokker--Planck equation associated with the mean-field dynamics of a particle system called the dispersion process. For different regimes of the average number of particles per site (denoted by $\mu > 0$), we establish various quantitative long-time convergence guarantees toward the global equilibrium (depending on the sign of $\mu - 1$), which is also confirmed by numerical simulations. The main novelty/contribution of this manuscript lies in the careful and tricky analysis of a nonlinear Volterra-type integral equation satisfied by a key auxiliary function.
\end{abstract}
	
\maketitle
\tableofcontents

\section{Introduction}

In this paper, we study the following one-dimensional discrete Fokker--Planck equation which models the evolution of the particle distribution $\bp (t) = \{p _n (t)\} _{n = 0} ^\infty$ associated with the mean-field dispersion process for $t > 0$ (see section \ref{subsec:1.1} for more details on the origin of the following dynamical system):
\begin{align}
	\label{eq:law_limit}
	\bp' (t) &= \mathcal D _+ [\{n \mathbf1 _{\{n \ge 2\}}\} _n \bp (t)] - \mathcal D _- [a (t) \bp (t)], & a (t) &= \sum _{n = 2} ^\infty n p _n (t).
\end{align}
Here $\bp: [0, \infty) \to \cS$ represents the probability mass function of the number of particles in a site, where
\begin{align*}
	\cS = \left\{
		\bq \in [0, 1] ^\mathbb N: \bq = \{q _n\} _{n = 0} ^\infty \text{ with } \sum _{n = 0} ^\infty q _n = 1
	\right\}
\end{align*}
is the set of probability distributions on $\mathbb N$. $\mathcal D _+/\mathcal D _-$ are the forward/backward difference operators, and $a: [0, \infty) \to \mathbb R$ represents the number of active particles per site. This is an infinite-dimensional dynamical system $\bp' (t) = \mathcal F [\bp (t)]$ where the generator $\mathcal F: \cS \to \mathbb R ^\mathbb N$ can be written down in components as the following:
\begin{align}
	\label{eq:L}
	\mathcal F [\bq] _n = \begin{cases}
	 -\left(\sum _{k \ge 2} k q _k \right) q _0 & n = 0,\\
	 2 q _2 - \left(\sum _{k \ge 2} k q _k \right) (q _1 - q _0) & n = 1, \\
	 (n + 1) q _{n + 1} - n q _n - \left(\sum _{k \ge 2} k q _k \right) (q _n - q _{n - 1}) & n \ge 2.
	\end{cases}
\end{align}

Our main result is summarized in the following theorem. Below, $W _0$ denotes the principal branch of the Lambert $W$ function \cite{lambert_Observationes_1758} (defined to be the unique solution $y = W_0(x)$ to the relation $y \expo ^{-y} = x$ whenever $x \geq -\expo ^{-1}$), and we employ the Japanese bracket notation $\langle t \rangle \coloneqq \sqrt{1 + t ^2}$.

\begin{theorem}
	\label{thm:main}
	Let $\bp ^0 = \{p ^0 _n\} _{n = 0} ^\infty \in \cS$, and denote its mean by $\mu = \sum _{n = 0} ^\infty n p ^0 _n$. If $\mu \le 1$, we assume in addition that its second moment is finite. Then there exists a positive constant $C$ depending only on the second moment of $\bp ^0$ when $\mu \le 1$, or on the first moment $\mu$ when $\mu > 1$, such that the solution $\bp (t)$ to \eqref{eq:law_limit} with initial value $\bp (0) = \bp ^0$ converges strongly as $t \to +\infty$ at the following rate.
	\begin{enumerate}[\upshape i)]
		\item If $0 < \mu < 1$, then
		\begin{align*}
			\left\lVert \bp (t) - \bp ^* \right\rVert _{\ell ^1} \le C \expo^{-2(1-\mu)t}.
		\end{align*}

		\item If $\mu = 1$, then
		\begin{align*}
			\left\lVert \bp (t) - \bp ^* \right\rVert _{\ell ^1} \le C\,t^{-1}.
		\end{align*}
	\end{enumerate}
	For $\mu > 1$, denote $\nu \coloneqq \mu + W _0 (-\mu \expo ^{-\mu})$. Then
	\begin{enumerate}[\upshape i)]
		\setcounter{enumi}{2}
		\item If $1 < \mu < \frac\expo{\expo - 1}$, then
		\begin{align*}
			\left\lVert \bp (t) - \bpb \right\rVert _{\ell ^1} \le C  \langle t \rangle ^\frac12 \expo ^{-\nu t}.
		\end{align*}

		\item If $\mu \ge \frac\expo{\expo - 1}$, then there exists $K > 0$ depending only on $\mu$ such that
		\begin{align*}
			\left\lVert \bp (t) - \bpb \right\rVert _{\ell ^1} \le C \langle t \rangle ^{K + \frac12} \expo ^{-t}.
		\end{align*}
	\end{enumerate}
	The equilibrium is given by the following:
	\begin{align}
		\label{eq:bernoulli}
		\bp ^* &= \{1 - \mu, \mu, 0, 0, \dots\}, \\
		\label{eq:zero_truncated_poisson}
		\bpb &= \{0, \pb _1, \pb _2, \pb _3, \dots\} \qquad \text{ with } \qquad
		\pb _n = \frac{\nu ^n}{n!} \cdot \frac1{\expo ^\nu - 1}, \qquad n \ge 1.
	\end{align}
\end{theorem}

The proof of this theorem is divided into Corollary \ref{cor:1} in Section \ref{sec:sec3} for $\mu \le 1$ and Corollary \ref{cor:2} in Section \ref{sec:sec4} for $\mu > 1$. We also emphasize that searching for sharp or optimal decay rates in the aforementioned results is out of the scope of this manuscript.

\subsection{Derivation of the model}
\label{subsec:1.1}

Equation \eqref{eq:law_limit} can be derived by resorting to a similar approach as taken in \cite{cao_sticky_2024} as the mean-field limit of the dispersion process, which describes the evolution of a stochastic particle system on a graph. The dispersion process is a continuous-time Markov process where $M$ indistinguishable particles are distributed and moving across the vertices of a complete graph with $N$ vertices, known as ``sites''. At any given time, each particle that shares a common site with at least one other particle is called ``active'', and it will move to a neighbor site, uniformly randomly chosen, at a prescribed fixed rate (e.g. by Poisson signal). Particles that reside alone on a site are called ``inactive'', and they will stay there until activated by an incoming particle. This scheme is illustrated in Figure \ref{fig:illustration_model}. The process will terminate when there are no remaining active particles. This model was first introduced in \cite{cooper_dispersion_2018} and further studied in \cite{de_dispersion_2023,frieze_note_2018,shang_longest_2020}, where the random time of termination is analyzed using probabilistic tools.

\begin{figure}[!htb]
	\centering
	\includegraphics[width=.67\textwidth]{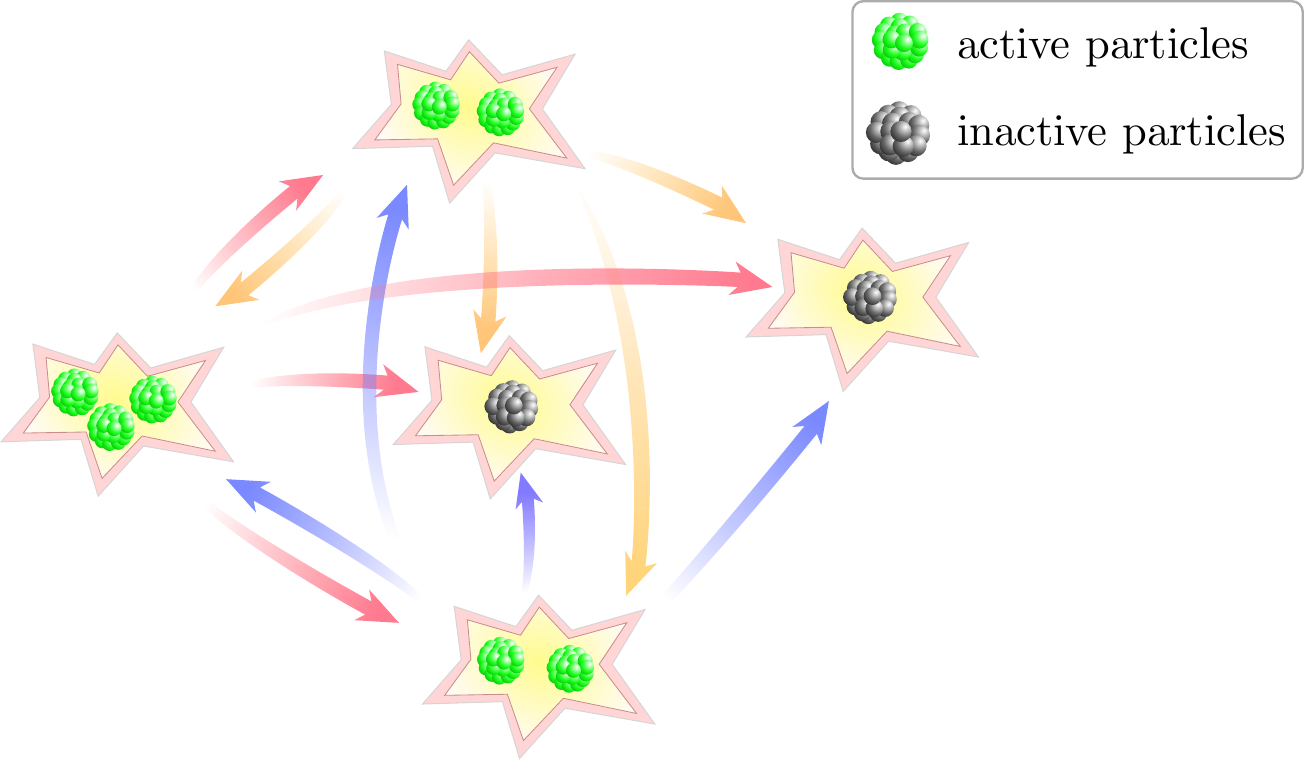}
	\caption{Illustration of the dispersion dynamics on a complete graph with $N = 5$ nodes/sites and $M = 9$ particles. Particles which share a common site will be ``active'' and move across sites.}
	\label{fig:illustration_model}
\end{figure}

The evolution of the particle distribution on the graph depends heavily on the ratio $\mu = M / N$, which represents the average number of particles per site. In the underpopulated scenario $\mu \le 1$, the process will terminate almost surely when all particles become inactive, but in the overpopulated case $\mu > 1$ the process will continue indefinitely.

Let us now analyze this process from the sites' perspective. Denote the number of particles on site $i$ at time $t$ by $X _i (t)$. Particle dynamics can be represented by the following update rule: For each pair $(i, j)$ of sites, if $X _i > 1$, then one particle will move from site $i$ to site $j$ at a rate proportional to $X _i$:
\begin{align}
	\label{model:dispersion}
	(X _i, X _j) \xlongrightarrow{\text{rate } X _i} (X _i - 1, X _j + 1) \qquad \text{ if } X _i > 1.
\end{align}
The rate at which site $i$ loses a particle is proportional to $X _i$, because every particle here is active. As a consequence, the rate at which any site gains a particle is proportional to the total number of active particles per site $\frac1N \sum _{i = 1} ^N X _i \mathbf1 _{\{X _i \ge 2\}}$. By taking the large population limit as $N, M \to +\infty$, we formally obtain the Fokker--Planck equation \eqref{eq:law_limit}. We will keep the scaling law $M / N \to \mu \in (0, +\infty)$, so that the law of $X _i$ will converge.

The continuous-time dispersion model we have described is a standard interacting particle system and is amenable to mean-field type analysis under the large population limit $N \to \infty$, which is detailed in a recent work \cite{cao_sticky_2024} on a related model. The rigorous justification of this transition from the stochastic interacting agents systems \eqref{model:dispersion} into the associated mean-field ODE system \eqref{eq:law_limit} requires the proof of the \emph{propagation of chaos} property \cite{sznitman_topics_1991}, which is beyond the scope of the present manuscript.

\subsection{Review on related models}

Here we detail several related models in the literature. One closely related model is the \emph{sticky dispersion process}, which is a variant of the dispersion process where each inactive particle stays permanently inactive. This model was studied in \cite{cao_sticky_2024}, where the mean-field limit was derived and the convergence to the equilibrium was established. The update rule is as follows:
\begin{align}
	(X _i, X _j) \xlongrightarrow{\text{rate } X _i - 1} (X _i - 1, X _j + 1) \qquad \text{ if } X_i > 1.
\end{align}
In the mean-field limit, the long-time equilibrium distribution is also the Bernoulli distribution for $\mu \le 1$, and for $\mu > 1$ the equilibrium is the modified Poisson distribution \cite{cao_sticky_2024}. The analysis is slightly simpler than the standard dispersion process, because the number of active particles $a (t)$ can be explicitly solved. Although the dispersion process investigated in this work resembles (at the microscopic level) to its sticky version studied in the companion work \cite{cao_sticky_2024} and the final asymptotic behavior might appear similar in one sense or another, the large time analysis of the mean-field ODE system \eqref{eq:law_limit} is significantly harder in the region where $\mu > 1$ compared to its sticky counterpart, largely due to a much tricker nonlinearity caused by $a(t) = \sum_{n\geq 2} n\,p_n$ in the equation \eqref{eq:law_limit} (this nonlinearity is more benign in the mean-field sticky dispersion model). For these reasons, section \ref{sec:sec4} constitutes the most technical part of this manuscript and requires significant novel ideas from the theory of PDEs.

A related particle system is the Becker--D\"oring model \cite{ball_becker-doring_1986,ball_discrete_1990,ball_asymptotic_1988,jabin_rate_2003,niethammer_evolution_2003,niethammer_scaling_2004}, which describes the evolution of the size distribution of clusters in a nucleation process. Each cluster may absorb or emit one particle at a rate depending on its size. Let $X _i$ be the number of particles in the $i$-th cluster, then the update rule is
\begin{align*}
	(X _i, X _j) \xlongrightarrow{\text{rate } a _{X _i}} (X _i + 1, X _j - 1) \qquad \text{ if } X _i > 1, X _j = 1 \\
	(X _i, X _j) \xlongrightarrow{\text{rate } b _{X _i}} (X _i - 1, X _j + 1) \qquad \text{ if } X _i > 1, X _j = 0.
\end{align*}
Here $(i, j)$ should be sampled from the set of all pairs of non-empty clusters and the first empty cluster, and the cluster-dependent jump rates $a _n, b _n \ge 0$ depend on the specific set-up of the model which we omit to describe at this moment. We will elaborate on the Becker--D\"oring model in section \ref{sec:sec2}.

Another example is the poor-biased exchange model \cite{cao_derivation_2021,cao_uniform_2024} studied in econophysics. It can be viewed as a variant of the dispersion process where every particle is always active, disregarding the number of particles on a site. The update rule is as follows:
\begin{align}
	(X _i, X _j) \xlongrightarrow{\text{rate } X _i} (X _i - 1, X _j + 1) \qquad \text{ if } X _i > 0.
\end{align}
In this and many other related econophysics models \cite{boghosian_h_2015,cao_biased_2023,cao_binomial_2022,cao_uncovering_2022,chakraborti_statistical_2000,chatterjee_pareto_2004,heinsalu_kinetic_2014,lanchier_rigorous_2017,lanchier_rigorous_2019,matthes_steady_2008}, particles are currencies which are exchanged by agents represented by nodes. Each agent passes a currency to another agent at a rate proportional to the amount of currency they have. Our model can be viewed as a variant of the poor-biased exchange model with a wealth-flooring policy, where agents with less than 2 currencies are not allowed to exchange.

We refer interested readers to \cite{chakraborti_econophysics_2011,dragulescu_statistical_2000,during_kinetic_2008,kutner_econophysics_2019,pereira_econophysics_2017,savoiu_econophysics_2013} for more information on econophysics. The aforementioned propogagion of chaos property was proved for many econophysics models \cite{cao_derivation_2021,cao_entropy_2021,cao_explicit_2021,cao_interacting_2022,cao_uniform_2024,cortez_quantitative_2016}.

\begin{figure}[htbp]
	\centering
	\includegraphics[width=.5\textwidth]{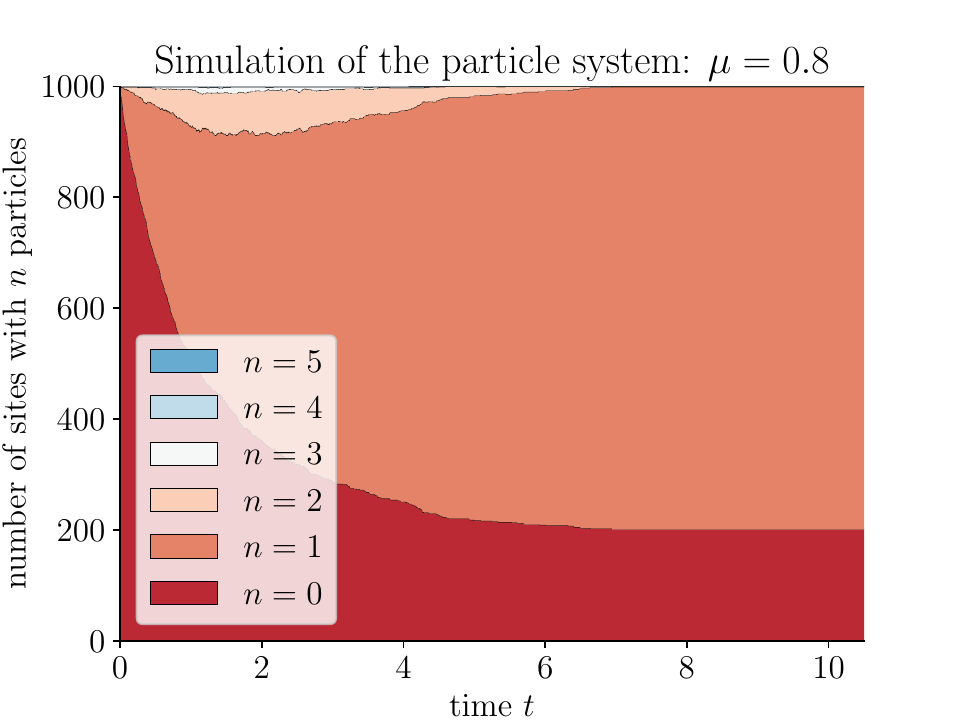}%
	\includegraphics[width=.5\textwidth]{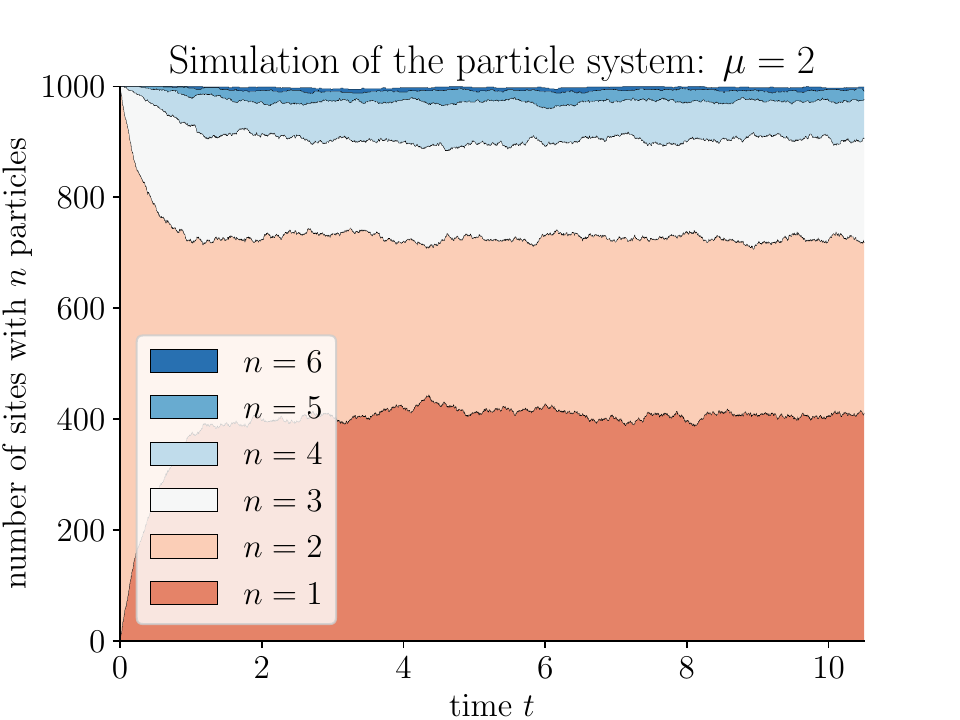}
	\caption{Stackplots of the agent-based simulation with $\mu = 0.8$ and $\mu = 2$ during the first $10$ units of time. At a given time $t$, the width of $n$-th layer represents the number of sites hosting $n$ particles, making a total of $N = 1,000$ sites. For $\mu = 0.8$, we initially put all particles at one single site; the distribution of particles converges to a Bernoulli distribution and no site hosts two or more particles (left). For $\mu = 2$, we put two particles in each site initially; each site hosts at least one particle and the distribution of particles stabilizes around a zero-truncated Poisson distribution after a few units of time (right).}
	\label{fig:agent_based_simulation}
\end{figure}

\subsection{Numerical simulations}

To verify the result, we perform two agent-based simulations on $N=1000$ sites, with $800$ and $2000$ particles respectively. We used two different initial values. For $M = 800$, all particles are initially placed at one single site; for $M = 2000$, they are evenly distributed on the graph at the beginning. We record the number of particles in each site $X _i$ over time. The distributions of $X _i$ in the first $10$ units of time are plotted in Figure \ref{fig:agent_based_simulation}.
We observe that for the underpopulated scenario $\mu = 0.8$, the distribution of the number of particles in a site converges to a Bernoulli distribution $\bp ^*$ defined in \eqref{eq:bernoulli} with mean $0.8$ after approximately 6.5 seconds (Figure \ref{fig:agent_based_simulation}-left). For the overpopulated case, it stabilizes near a zero-truncated Poisson distribution $\bpb$ with mean $2$ defined in \eqref{eq:zero_truncated_poisson} (Figure \ref{fig:agent_based_simulation}-right). To furthermore confirm the convergence, we continue the simulation for $2000$ time units and the bar plot of $X _i$'s is presented in Figure \ref{fig:agent_based_simulation_bar}.
\begin{figure}[htbp]
	\centering
	\includegraphics[width=.5\textwidth]{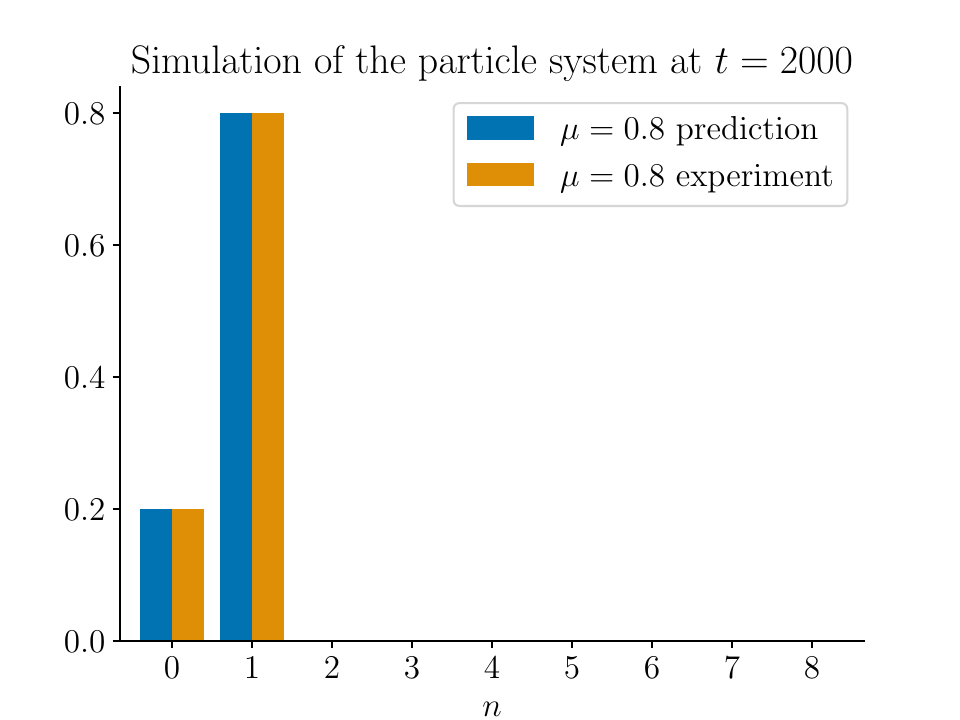}%
	\includegraphics[width=.5\textwidth]{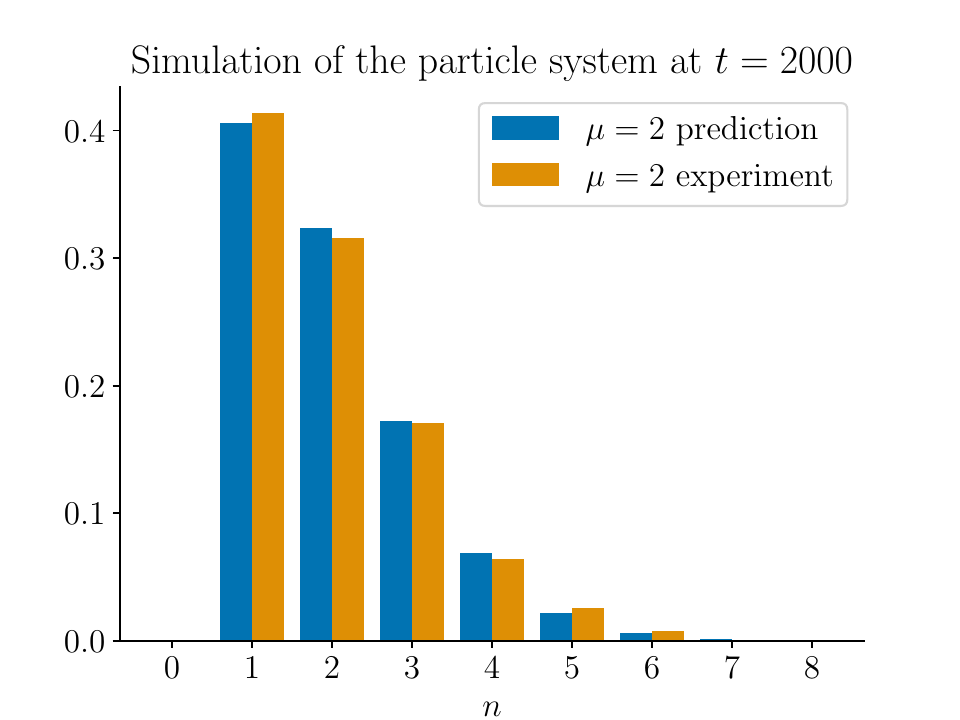}
	\caption{Distribution of particles for the dispersion model with $N = 1,000$ agents after $2,000$ units of time, using two different values of $\mu$. For $\mu = 0.8$, the final distribution coincides exactly with the Bernoulli distribution with mean $0.8$, where we put all the particles into a single site initially. For $\mu = 2$, the terminal distribution is well-approximated by a zero-truncated Poisson distribution \eqref{eq:zero_truncated_poisson}, where we put $X_i(0) = \mu$ for all $1\le i\le N$ initially.}
	\label{fig:agent_based_simulation_bar}
\end{figure}

Next, we numerically solve the infinite system \eqref{eq:L} by truncating the first $101$ components (i.e. $(p _0 (t),\dots, p _{100} (t))$). We again choose the parameters $\mu = 2$ and $\mu = 0.8$. For $\mu = 0.8$, we use $p _{100} (0)= \mu / 100$, $p _0 (0) = 1 - \mu / 100$, and $p _i (0) = 0$ on other components for initial condition. For $\mu = 2$ we use $p _2 (0) = 1$ and $p _i (0) = 0$ on other components. We apply the standard Runge--Kutta fourth-order scheme and the time step is set to $\Delta t = 0.01$. We plot the value of $p _n (t)$ as a stackplot in Figure \ref{fig:ODEsimulation}, and we observe in both cases convergence to equilibrium distributions. Corresponding results for agent-based simulation in Figure \ref{fig:agent_based_simulation} are plotted here in yellow for comparison. We find the numerical solution to the ODE system \eqref{eq:law_limit} is in good agreement with the agent-based simulation results, which suggests \eqref{eq:law_limit} is a valid mean-field limit.

\begin{figure}[htbp]
	\centering
	\includegraphics[width=0.5\textwidth]{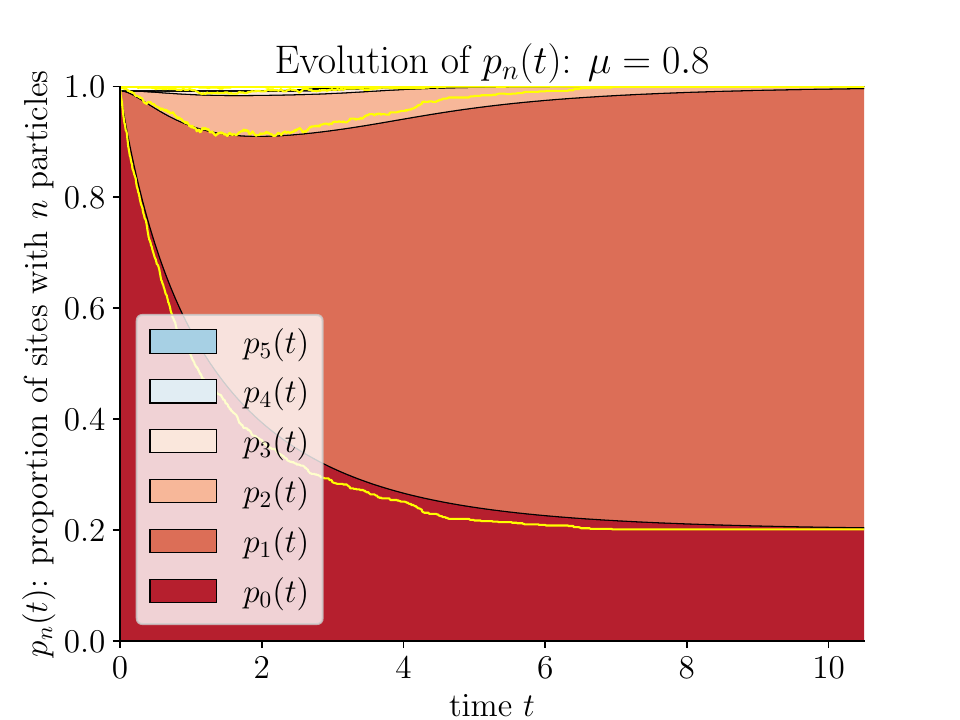}%
	\includegraphics[width=0.5\textwidth]{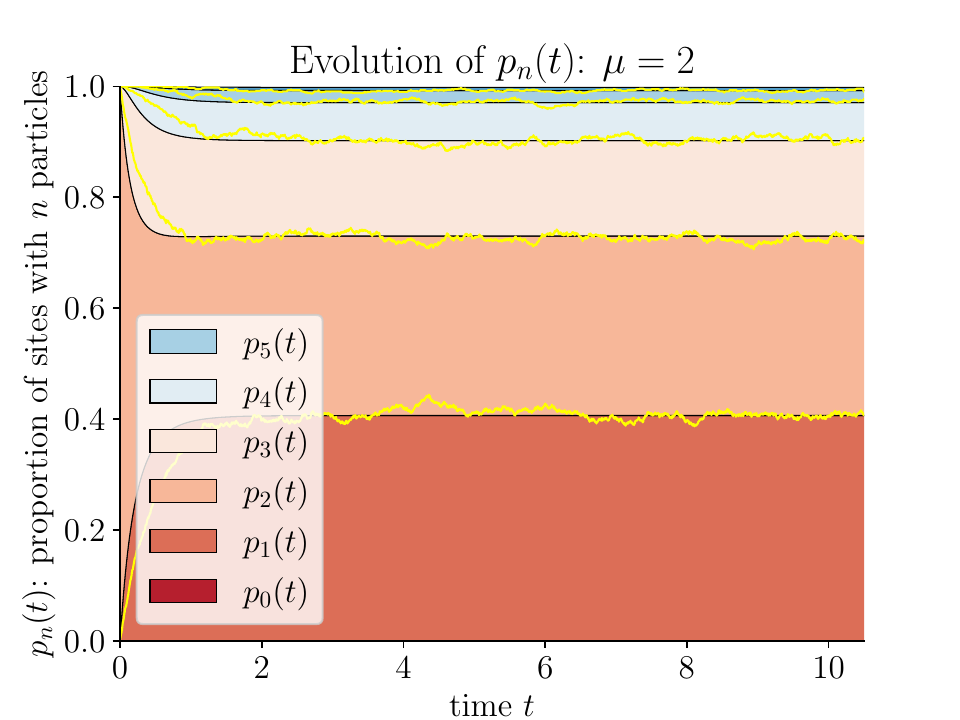}
	\caption{Stackplot of the numerical solution to the truncated ODE system $\{p _n (t)\} _{n = 0} ^{100}$ with $\mu = 0.8$ and $\mu = 2$ during the first 10 units of time. At a given time $t$, the width of $n$-th layer represents $p _n (t)$ which sum up to $1$. For $\mu = 0.8$ the distibution of particles converges to a Bernoulli distribution with mean $\mu$ (left). For $\mu = 2$, the distribution converges to the zero-truncated Poisson distribution with mean $\mu$ (right). The overlayed yellow lines represent the corresponding agent-based simulation results.}
	\label{fig:ODEsimulation}
\end{figure}

\subsection{Organization of the paper}

In section \ref{sec:sec2}, we show some basic properties of the system \eqref{eq:law_limit} and locate the equilibrium distribution. The form of the equilibrium depends on the parameter $\mu \in (0, \infty)$ which represents the average amount of particles per site initially. For the underpopulated case $\mu \in (0, 1]$, we prove the solution converges to the Bernoulli distribution $\bp ^*$ using a Lyapunov functional in section \ref{sec:sec3}. In section \ref{sec:sec4}, we analyze the overpopulated case $\mu > 1$ and show that the solution relaxes to the zero-truncated Poisson distribution $\bpb$ at an exponential rate. The proof is based on the analysis of the probability generating function (PGF) of the solution, which satisfies a transport equation. Finally, we draw a conclusion in section \ref{sec:sec5}.

\section{Elementary properties of the ODE system}
\label{sec:sec2}

After we achieved the transition from the interacting agents system \eqref{model:dispersion} to the deterministic nonlinear ODE system \eqref{eq:law_limit}, our main goal is to show convergence of solution of \eqref{eq:law_limit} to its (unique) equilibrium solution. We aim to describe some elementary properties of solutions of \eqref{eq:law_limit} in this section. As we have indicated in the introduction, the large time behavior of solutions to \eqref{eq:law_limit} depends critically on the range to which the parameter $\mu$ belongs. Before we dive into the detailed analysis of the system of nonlinear ODEs, we first establish some preliminary observations regarding solutions of \eqref{eq:law_limit}.

As we want $a (t)$ to be finite, we equip $\cS$ with the following weighted $\ell ^1$ norm:
\begin{align*}
	\left\lVert \bq \right\rVert _{\cS} := \sum _{n = 0} ^\infty \left(1 + n \right) |q _n|, \qquad \bq \in \mathbb R ^{\mathbb N}.
\end{align*}
Clearly $\left\lVert \bq \right\rVert _{\ell ^\infty} \le \left\lVert \bq \right\rVert _{\cS}$. It can be seen from \eqref{eq:L} that $\left\lVert \mathcal F [\bq] \right\rVert _{\ell ^\infty} \le C \left\lVert \bq \right\rVert _{\cS}$. Therefore, we say $\bp: [0, \infty) \to \mathcal S$ is a \textbf{(classical) solution} to the system \eqref{eq:law_limit} if 
\begin{align*}
	\bp \in C ([0, \infty); \cS) \cap C ^1 ([0, \infty); \ell ^\infty)	
\end{align*}
and it verifies $\bp ' (t) = \mathcal F [\bp (t)]$. 

\begin{lemma}
	\label{lem:1}
	If $\bp (t)$ is a solution to the system \eqref{eq:law_limit}, then
	\begin{align}
		\label{eq:conservation_mass_mean_value}
		\sum _{n = 0} ^\infty \mathcal F [\bp (t)] _n = 0 \quad \text{and} \quad \sum _{n = 0} ^\infty n \mathcal F [\bp (t)] _n =0.
	\end{align}
	In particular, the total probability mass and the average amount of particles per site are conserved.
\end{lemma}

The proof of Lemma \ref{lem:1} is based on straightforward computations and will be skipped. Thanks to these conservation relations, the solution $\bp (t)$ lives in the space of probability distributions on $\mathbb N$ with the prescribed mean value $\mu$, defined by
\begin{align}
	\label{eq:S _mu}
	\cS _\mu \coloneqq \left\{
	\bq \in [0, 1] ^\mathbb N: \bq = \{q _n\} _{n = 0} ^\infty \text{ with } \sum _{n = 0} ^\infty q _n = 1, \sum _{n = 0} ^\infty n q _n = \mu
	\right\}.
\end{align}
We can use Galerkin method (see \cite[section 8.6]{arbogast_2024}) to show existence of a classical solution. The idea is to show the solutions to the truncated finite-dimensional ODEs are uniformly bounded in $C ([0, T); \mathcal S)$, so a weak limit (say in $L ^2 (0, T; \ell ^2)$) can be extracted. One then verifies the weak limit is indeed a classical solution. As for uniqueness, we provide a proof in Appendix \ref{app:uniqueness} for interested readers. More importantly, if the initial value has mean $\mu$, then the number of active particles per site $a(t)$ is given by
\begin{align*}
	a (t) = \sum _{n = 2} ^\infty n p _n (t) = \mu - p _1 (t).
\end{align*}
The system \eqref{eq:law_limit} will be equivalent to the following system of nonlinear ODEs:
\begin{align}
	\label{eq:law_limit_repeat}
	\bp' (t) &= \mathcal D _+ [\{n \mathbf1 _{\{n \ge 2\}}\} _n \bp (t)] - \mathcal D _- [(\mu - p _1 (t)) \bp (t)].
\end{align}
Writing in components, we have $\bp' (t) = \mathcal F [\bp (t)]$ where the generator $\mathcal F: \cS _\mu \to \mathbb R ^\mathbb N$ is given by
\begin{align}
	\label{eq:L_rewrite}
	\mathcal F [\bq] _n = \begin{cases}
		-(\mu - q _1) q _0 & n = 0,\\
		2 q _2 - (\mu - q _1) (q _1 - q _0) & n=1, \\
		(n + 1) q _{n + 1} - n q _n - (\mu - q _1) (q _n - q _{n - 1}) & n \ge 2.
	\end{cases}
\end{align}

The Fokker--Planck type equation \eqref{eq:law_limit_repeat} admits a heuristic interpretation as a jump process with loss and gain, and we illustrate this perspective via Figure \ref{fig:illustration_rates} below. Recall that $a (t) = \mu - p _1 (t)$ is the number of active particles per site.
\begin{figure}[htb]
	\centering
	\includegraphics{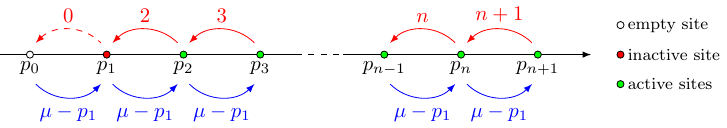}
	\caption{Schematic illustration of the Fokker--Planck type system of nonlinear ODEs \eqref{eq:law_limit_repeat} as a jump process with loss and gain.}
	\label{fig:illustration_rates}
\end{figure}

\begin{remark}
	\label{rem:becker-doring}
	The system \eqref{eq:law_limit_repeat} also resembles another system of nonlinear ODEs known as the Becker--D\"oring cluster equations. Following the convention, if $p _n$ represents the number of clusters with $n$ particles, $\mu$ is the total number of particles, then $\bp = \{p _n\} _{n = 1} ^\infty$ satisfies $\bp' (t) = \mathcal{BD} (\bp (t))$ where the generator $\mathcal{BD}: \cS _\mu \to \mathbb R ^{\mathbb N ^*}$ is
	\begin{align*}
		\mathcal {BD} [\bq] _n = \begin{cases}
			(b _{n + 1} q _{n + 1} - b _n q _n) - q _1 (a _n q _n - a _{n - 1} q _{n - 1}) & n \ge 2,\\
			b _2 q _2 - a _1 q _1 ^2 + \sum _{n = 1} ^\infty (b _{n + 1} q _{n + 1} - q _1 a _n q _n) & n = 1.
		\end{cases}
	\end{align*}
	Here ${\bf a} = \{a _n\} _{n = 1} ^\infty$, ${\bf b} = \{b _n\} _{n = 1} ^\infty$ are given positive constants, with $b _1 = 0$. We can also write the Becker--D\"oring equations as a discrete Fokker--Planck type equation:
	\begin{align*}
		\bp' (t) &= \mathcal D _+ [{\bf b} \bp (t)] - \mathcal D _- [p _1 (t) {\bf a} \bp (t)] + {\bf f} (t), \\
		{\bf f} (t) &= (({\bf b} - p _1 (t) {\bf a}) \cdot \bp, 0, 0, \dots).
	\end{align*}
Here $\bf b \bp$ is the pointwise product and ``$\cdot$'' denotes the inner product. Compared with \eqref{eq:law_limit_repeat}, the generator of Becker--D\"oring equations $\mathcal {BD}$ is also a second-order difference operator; upward drift ${\bf b}$ is constant, whereas downward drift $p _1 {\bf a}$ depends only on $p _1$. However, the total mass here is not conserved due to a quadratic reaction term $\bf f$, which makes Becker--D\"oring more difficult to analyze. Nevertheless, both $\mathcal F$ and $\mathcal{BD}$ are linear in $(q _2, q _3, \dots)$. We refer interested readers to \cite{ball_becker-doring_1986,ball_asymptotic_1988,ball_discrete_1990,jabin_rate_2003,niethammer_evolution_2003,niethammer_scaling_2004} and references therein.
\end{remark}

Next, we identify the unique equilibrium solution associated with the system \eqref{eq:law_limit_repeat}.

\begin{proposition}
	\label{prop:equilibrium}
	The unique equilibrium solution of \eqref{eq:law_limit_repeat} in the space $\cS _\mu$, for $\mu \in (0, 1]$, is given by $\bf p ^*$ defined in \eqref{eq:bernoulli}. The unique equilibrium solution of \eqref{eq:law_limit_repeat} in the space $\cS _\mu$, when $\mu \in (1,\infty)$, is provided by $\bpb$ defined in \eqref{eq:zero_truncated_poisson}.
\end{proposition}

\begin{proof}
	To find equilibrium point of the system \eqref{eq:law_limit_repeat}, we need to solve $\mathcal F (\bq) = 0$, which reduces to the following by induction:
	\begin{align}
		\label{eq:identity-1}
		(\mu - q _1) q _0 &= 0, \\
		\label{eq:identity-2}
		(\mu - q _1) q _n &= (n + 1) q _{n + 1} \qquad \text{for } n \ge 1.
	\end{align}
	On the one hand, if $\mu = q _1 \le 1$, then by \eqref{eq:identity-2} $q _n = 0$ for all $n \ge 2$, and we deduce that the unique equilibrium solution is $\bq = \bp ^*$ with
	\begin{align*}
		p ^* _0 &= 1 - \mu, &
		p ^* _1 &= \mu, &
		p^*_n &= 0 \qquad \text{for } n \ge 2.
	\end{align*}
	Note that in the case $\mu = 1$ the Bernoulli distribution \eqref{eq:bernoulli} degenerates to the Dirac delta distribution centered at $1$.
	On the other hand, for $\mu > 1 \ge p _1$, we deduce from \eqref{eq:identity-1} that $p _0 = 0$, and the unique equilibrium distribution is $\bq = \bpb$ with
	\begin{align*}
		\pb _0 &= 0, &
		\pb_n &= \frac{(\mu - \pb _1) ^{n - 1}}{n!} \pb _1 \qquad \text{for } n \ge 1
	\end{align*}
	where $\pb _1 > 0$ is chosen such that $\bpb \in \cS _\mu$. Since $\sum _{n = 0} ^\infty \pb _n = 1$, we deduce that $(\mu - \pb _1) / \pb _1 = \sum _{n = 1} ^\infty (\mu - \pb _1) ^n / n! = \expo ^{\mu - \pb _1} - 1$, thus $\pb_1 \expo^{-\pb _1} = \mu \expo ^{-\mu}$, from which we deduce $\pb _1 = - W _0 (-\mu \expo ^{-\mu})$ because $\pb _1 < \mu$. We finish the proof by introducing a new constant $\nu = \mu - \bar p _1$.
\end{proof}

\begin{remark}
	The zero-truncated Poisson distribution $\bpb$ defined in \eqref{eq:zero_truncated_poisson} admits a simple interpretation in terms of random variables. Indeed, if $X \sim \operatorname{Poisson} (\nu)$, then the distribution of $X$ conditioned on $X \ge 1$ obeys the zero-truncated Poisson distribution, whose law is given by $\bpb$.
\end{remark}

\section{Underpopulated case: Convergence to the Bernoulli distribution}
\label{sec:sec3}

To justify the large time convergence of solutions of the system \eqref{eq:law_limit_repeat} when $\mu \in (0, 1]$, we employ a suitable Lyapunov functional associated to the dynamics \eqref{eq:law_limit_repeat}. For this purpose, we define the following energy functional:
\begin{align}
	\label{eq:energy}
	\cE [\bq] \coloneqq \sum _{n \ge 0} n ^2 q _n - \mu \qquad \bq \in \cS _\mu,
\end{align}
which is just a shifted version of the second (raw) moment of the distribution $\bq$. We first start with an elementary variational characterization of the Bernoulli distribution $\bp ^*$.

\begin{lemma}
	\label{lem:Bernoulli_characterization}
	For each $\mu \in (0, 1]$, the Bernoulli distribution $\bp^*$ with parameter $\mu$ satisfies
	\begin{align}
		\label{eq:bernoulli_characterization}
		\bp ^* = \argmin _{\bq \in \cS _\mu} \sum _{n \ge 0} n ^2 q _n.
	\end{align}
	Consequently, $\cE [\bq] \ge 0$ for all $\bq \in \cS _\mu$ and the equality holds if and only if $\bq = \bp ^*$.
\end{lemma}

\begin{proof}
	For $\bq \in \cS _\mu$, we have $\sum _{n \ge 0} n q _n = \mu$ and thus $\sum _{n \ge 1} n ^2 q _n \ge \sum _{n \ge 1} n q _n = \mu$, in which the inequality will become an equality if and only if $q _n = 0$ for all $n \ge 2$. This finishes the proof of Lemma \ref{lem:Bernoulli_characterization}.
\end{proof}

We now prove the following quantitative convergence result for the dissipation of $\cE [\bp (t)]$ along solutions to the system of nonlinear ODEs \eqref{eq:law_limit_repeat} when $\mu \in (0, 1]$.

\begin{theorem}
	\label{thm:1}
	Assume that $\bp (t)$ is the classical solution to the system \eqref{eq:law_limit_repeat} with $\bp (0) \in \cS _\mu$ and $\mu \in (0, 1]$, then for all $t\ge 0$ we have
	\begin{align}
		\label{eq:energy_decay_mu<1}
		\cE [\bp (t)] \le \cE [\bp (0)] \expo ^{-2 (1 - \mu) t} \qquad \text{ when } \mu < 1,
	\end{align}
	and
	\begin{align}
		\label{eq:energy_decay_mu=1}
		\cE [\bp (t)] \le \cE [\bp (0)]\expo^{-2t} + \frac{4}{t+2/p _0(0)} + 2p _0(0)\expo^{-t} \qquad \text{when } \mu = 1.
	\end{align}
\end{theorem}

\begin{proof}
Assume $\cE [\bp (0)] < +\infty$ otherwise the conclusion holds trivially. A straightforward computation gives us
\begin{align}
	\label{eq:evolution_of_E}
	\begin{aligned}
		\frac{\dd}{\dd t} \cE [\bp] &= 2 p _2 + (\mu - p _1)(p _0 - p _1) \\
		& \qquad + \sum _{n \ge 2} n ^2 \left[
			(n + 1) p _{n + 1} - n p _n - (\mu - p _1)(p _n - p _{n - 1})
		\right] \\
		&= 2 p _2 + (\mu - p _1)(p _0 - p _1) \\
		& \qquad + \sum _{n \ge 2} \left[
			n ^2 (n + 1) p _{n + 1} - n ^3 p _n
		\right] - (\mu - p _1) \sum _{n \ge 2} n ^2 (p _n - p _{n - 1}) \\
		&= 2 p _2 + (\mu - p _1)(p _0 - p _1) \\
		& \qquad + \left(
			p _1 - 2 p _2 + \mu - 2 \sum _{n \ge 0} n ^2 p _n
		\right) - (\mu - p _1)(p _0 - p _1 - 1 - 2 \mu) \\
		&= p _1 + \mu - 2 \sum _{n \ge 0} n ^2 p _n + (\mu - p _1)(1 + 2 \mu) \\
		&= -2 \cE [\bp] + 2 \mu (\mu - p _1).
	\end{aligned}
\end{align}
We observe that $0 \le \mu - p _1 = \sum _{n \ge 2} np _n \le \sum _{n \ge 2} (n ^2 - n)p _n = \cE [\bp]$ for all $\bp \in \cS _\mu$, thus for $\mu \in (0, 1)$ we deduce from \eqref{eq:evolution_of_E} that
\[\frac{\dd}{\dd t} \cE [\bp] \le -(2-2\mu)\cE [\bp],\] from which the exponential decay of $\cE [\bp (t)]$ \eqref{eq:energy_decay_mu<1} follows immediately. On the other hand, for $\mu = 1$, we have $2 p _0 + p _1 \ge 1$ since $\bp \in \cS _1$. The differential equation satisfied by $p _0$ implies that $p _0 ' = -(1 - p _1) p _0 \le -p ^2 _0$, whence
\begin{align*}
	p _0 (t) \le \frac{1}{t + 1 / p _0 (0)} \text{ and } p _1 (t) \ge 1 - \frac{2}{t + 1 / p _0 (0)}
\end{align*}
for all $t \ge 0$. Consequently, we derive from \eqref{eq:evolution_of_E} the following differential inequality:
\begin{align*}
	\frac{\dd}{\dd t} \cE [\bp] \le -2 \cE [\bp] + \frac{4}{t + 1 / p _0(0)},
\end{align*}
whence
\begin{align}
	\label{eq:energy_inequality}
	\cE [\bp (t)] \le \cE [\bp (0)]\expo^{-2t} + 4\expo^{-2t}\int_0^t \frac{\expo^{2s}}{s+ 1/p _0(0)} \dd s.
\end{align}
To conclude the proof and reach the advertised upper bound \eqref{eq:energy_decay_mu=1}, it suffices to notice that
\begin{align*}
	\int _0 ^t \frac{\expo ^{2s}}{s + 1 / p _0 (0)} \dd s
	&= \int _0 ^\frac t2 \frac{\expo ^{2s}}{s + 1 / p _0 (0)} \dd s + \int _\frac t2 ^t \frac{\expo ^{2s}}{s + 1 / p _0 (0)} \dd s \\
	&\le p _0 (0) \int _0 ^\frac t2 \expo ^{2s} \dd s + \frac1{\frac t2 + \frac1{p _0 (0)}} \int _\frac t2 ^t \expo ^{2s} \dd s \\
	&\le \frac 12 p _0 (0) \expo ^t + \frac{\expo^{2t}}{t + 2 / p _0 (0)}
\end{align*}
for all $t \ge 0$. Thus the proof of Theorem \ref{thm:1} is completed.
\end{proof}

To illustrate the decay of the energy $\cE $ numerically, we use the same set-up as in the previous experiment shown in Figure \ref{fig:energy_decay} for two different values of $\mu \in (0, 1]$, using the semi-log scale.

\begin{figure}[htbp]
		\centering
		\includegraphics[width=.75\textwidth]{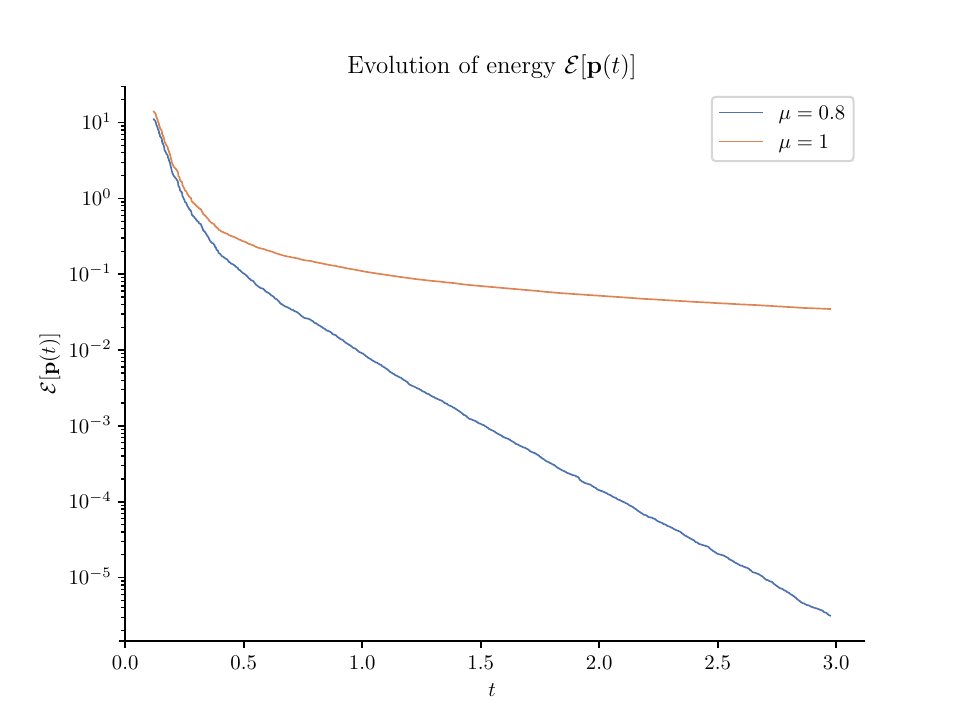}
		\caption{Evolution of the energy $\cE [\bp (t)]$ over $0\le t \le 3$ with $\mu = 0.8$ and $\mu = 1$. It can be seen from the picture that the energy decays exponentially for $\mu = 0.8$ with rate $C\expo ^{-0.4 t}$. For $\mu = 1$ the decay is slower.}
		\label{fig:energy_decay}
\end{figure}

As an immediate corollary, we can readily deduce the following strong convergence in $\ell^1$.

\begin{corollary}
	\label{cor:1}
	Under the settings of Theorem \ref{thm:1}, if $\bp (0)$ has a finite variance, then there exists some constant $C > 0$ depending only on $\mu$ and the initial datum $\bp (0)$ such that for all $t > 0$, it holds that
	\begin{align}
		\label{eq:ell1_decay_mu<1}
		\|\bp (t) - \bp ^*\| _{\ell ^1} \le C \expo^{-2 (1 - \mu) t} \qquad \text{when } \mu < 1,
	\end{align}
	and
	\begin{align}
		\label{eq:ell1_decay_mu=1}
		\|\bp (t) - \bp ^*\| _{\ell ^1} \le \frac Ct \qquad \text{when } \mu = 1.
	\end{align}
\end{corollary}

\begin{proof}
	We only prove the bound \eqref{eq:ell1_decay_mu=1} when $\mu = 1$ as the other bound \eqref{eq:ell1_decay_mu<1} which is valid for $\mu \in (0, 1)$ can be handled in a pretty similar way. Notice that
	\begin{align*}
		\cE [\bp (t)] = \sum _{n \ge 0} (n ^2 - n) p _n (t) = \sum _{n \ge 2} (n ^2-n)p _n(t)
	\end{align*}
	and $n ^2 \le 2(n ^2-n)$ for all $n \ge 2$, hence $\sum _{n \ge 2} n ^2p _n(t) \le 2\cE [\bp (t)]$.
	Therefore, we deduce that
	\begin{align*}
		\|\bp (t) - \bp ^*\| _{\ell ^1} &= |p _0 (t)| + |1 - p _1(t)| + \sum _{n \ge 2} |p _n (t)| \\
		&= 1 - 2 p _1 (t) + \sum _{n = 0} ^\infty p _n (t) = 2(1 - p _1(t)).
	\end{align*}
	Moreover, from the definition of $\cE$ we know $\cE [\bp] = p _1 + \sum _{n \ge 2} n ^2 p _n - \mu$, so
	\begin{align*}
		1 - p _1(t)
		&= \sum _{n \ge 2} n ^2p _n(t) - \cE [\bp (t)] \le \cE [\bp (t)] \le \frac{C}{t}
	\end{align*}
	for some constant $C > 0$ depending on $\bp (0)$ and $\mu$.
\end{proof}

We comment here that the proof of Corollary \ref{cor:1} implies an inequality of the form $\|\bp (t) - \bp ^*\| _{\ell ^1} \leq C\,\cE [\bp]$ for some generic constant $C > 0$, which is reminiscent of a Csiszár--Kullback--Pinsker type inequality in the study of certain convex entropy-like Lyapunov functionals (see for instance \cite{jungel_entropy_2016}).

\begin{remark}
It is worth mentioning that the Lyapunov functional (or ``energy'') $\cE [\bq]$ employed in the large time analysis of the system \eqref{eq:law_limit_repeat} when $\mu \in (0, 1]$ differs drastically from the Lyapunov functional constructed in \cite{cao_sticky_2024} for the large time analysis of the mean-field sticky dispersion model when $\mu \in (0,1]$. In fact, the Lyapunov functional discovered in the study of the mean-field sticky dispersion model \cite{cao_sticky_2024} in the region where $\mu \in (0,1]$ is the so-called Gini index which can be utilized to prove convergence in terms of the Wasserstein distance $W_1$ \cite{cao_gini_2024} (which is certainly weaker than the $\ell^1$ convergence reported here). To some extent, it demonstrates the necessity to significantly modify the appropriate ``energy'' even the underlying mean-field dynamical systems appear to be similar from a structural perspective.
\end{remark}

\section{Overpopulated case: Relaxation to zero-truncated Poisson distribution}
\label{sec:sec4}

We use a different approach to study the system \eqref{eq:law_limit_repeat} when $\mu > 1$. Treating the active particles per site $a (t)$ as a known function, we first show that the probability generating function solves a first order partial differential equation (PDE), which turns out to be an explicitly solvable transport equation. By writing out the solution, we show that an auxiliary function
\begin{align}
	\label{eq:auxillary}
	v(t) \coloneqq \exp \left(\int _0 ^t \expo ^{-t + s} a (s) \dd s\right)
\end{align}
must satisfy a nonlinear Volterra-type integral equation \cite{gripenberg_volterra_1990}. We then study this integral equation to extract convergence and convergence rate, which further sheds light on the convergence of the distribution $\bp(t)$ to the zero-truncated Poisson distribution $\bpb$ \eqref{eq:zero_truncated_poisson}.

\subsection{Probability generating function}
\label{subsec:4.1}

\newcommand{\G}{G}
\newcommand{\g}{g}

Define the probability generating function $\G: [0, +\infty) \times [-1, 1]$ of the solution $\bp(t)$ to \eqref{eq:law_limit} by
\begin{align*}
	\G (t, z) = \sum _{n = 0} ^\infty p _n (t) z ^n.
\end{align*}
Since $p _n (t) \ge 0$ and $\sum _{n = 0} ^\infty p _n (t) = 1$, we know the above series is absolutely summable. Moreover, because $\sum _{n = 0} ^\infty n p _n (t) = \mu$, we know that
\begin{align*}
	\frac{\partial \G}{\partial z} (t, z) = \sum _{n = 1} ^\infty n p _n (t) z ^{n-1}
\end{align*}
is absolutely summable. The ODE system \eqref{eq:law_limit} can thus be written as the following PDE for $\G$:
\begin{align}
	\label{eq:PDE}
	\partial _t \G &= (1 - z) [\partial _z \G - p _1 (t) - a (t) \G] = (1 - z) [\partial _z \G - \mu + a (t) - a (t) \G].
\end{align}
Recall that $a (t) = \mu - p _1 (t)$. We also recall that the probability generating function can recover the following statistics:
\begin{align*}
	& \dfrac{\partial ^k}{\partial z ^k} \G (t, 0) = k! p _k (t) \\
	& \dfrac{\partial ^k}{\partial z ^k} \G (t, 1) = \sum _{n = 0} ^\infty n (n - 1) \cdots (n - k + 1) p _n (t).
\end{align*}
Moreover, since $p _k (t) \ge 0$ for all $k$, we have monotonicity in all derivatives:
\begin{align*}
	\dfrac{\partial ^k}{\partial z ^k} \G (t, 0) < \dfrac{\partial ^k}{\partial z ^k} \G (t, z) < \frac{\partial ^k}{\partial z ^k} \G (t, 1), \qquad (t, z) \in [0, \infty) \times (0, 1).
\end{align*}

We first solve \eqref{eq:PDE} in terms of $p _1 (t)$ using the method of characteristics \cite{evans_partial_2022}.

\begin{lemma}
	\label{lem:solution-to-pde}
	The probability generating function $\G$ can be expressed using the following explicit formula: For $z \in \mathbb C$ with $|1 - z| \le 1$, $t \ge 0$, we have
	\begin{align*}
		\G (t, 1 - z)
		&= 1 + \left(
			\G (0, 1 - z \expo ^{-t}) - 1 - \mu z \int _0 ^t [v (s)] ^{z \expo ^{-t + s}} \expo ^{-t + s} \dd s
		\right) [v (t)] ^{-z},
	\end{align*}
	where $v: [0, \infty) \to \mathbb R$ is defined by \eqref{eq:auxillary}.
\end{lemma}

\begin{proof}
	We only prove for $z \in \mathbb R$ with $0 \le z \le 2$. This will be sufficient because both sides are analytic in $z$ and the equality follows by identity theorem.

	Define $\gamma (s) = 1 - z \expo ^{-t + s}$, then
	\begin{align*}
		\gamma (0) &= 1 - z \expo ^{-t}, &
		\gamma (t) &= 1 - z, &
		\gamma' (s) = -z \expo ^{-t + s} = -(1 - \gamma (s)).
	\end{align*}
	Now we let $g (s) = \G (s, \gamma (s)) - 1$, then the evolution of $\g$ satisfies
	\begin{align*}
		\g' (s) &= \partial _t \G (s, \gamma (s)) + \gamma' (s) \partial _z \G (s, \gamma (s)) \\
		&= [\partial _t - (1 - \gamma (s)) \partial _z] \G (s, \gamma (s)) \\
		&= - (1 - \gamma (s)) [a (s) \G (s, \gamma (s)) - a (t) + \mu] \\
		&=- z \expo ^{-t + s} [a (s) \g (s) + \mu],
	\end{align*}
	with
	\begin{align*}
		\g (0) = \G (0, \gamma (0)) - 1 = \G (0, 1 - z \expo ^{-t}) - 1.
	\end{align*}
	So $\g$ satisfies the following first order linear ODE with the above initial condition:
	\begin{align*}
		\g' (s) + z \expo ^{-t + s} a (s) \g (s) &= -\mu z \expo ^{-t + s}.
	\end{align*}
	Setting
	\begin{align}
		\label{eq:H}
		H (t) \coloneqq \int _0 ^t \expo ^s a (s) \dd s,
	\end{align}
	we have $v (t) = \expo ^{\expo ^{-t} H (t)}$ and
	\begin{align*}
		\g' (s) + z \expo ^{-t} H' (s) \g (s) &= -\mu z \expo ^{-t + s}, \\
		\frac{\dd}{\dd s} \left(\g (s) \expo ^{z \expo ^{- t} H (s)}\right) &= -\mu z \expo ^{-t + s} \expo ^{z \expo ^{- t} H (s)}, \\
		\frac{\dd}{\dd s} \left(\g (s) [v (s)] ^{z \expo ^{-t + s}}\right) &= -\mu z [v (s)] ^{z \expo ^{-t + s}} \expo ^{-t + s}.
	\end{align*}
	Integrating from $s=0$ to $s=t$ and using $v(0)=1$, we deduce that
	\begin{align*}
		\g (t) [v (t)] ^{z} &= g (0) - \mu z \int _0 ^t [v (s)] ^{z \expo ^{-t + s}} \expo ^{-t + s} \dd s.
	\end{align*}
	Finally, notice that $\g (t) = \G (t, \gamma (t)) - 1 = \G (t, 1 - z) - 1$ and $\g (0) = \G (0, 1 - z \expo ^{-t}) - 1$, we conclude the proof of Lemma \ref{lem:solution-to-pde}.
\end{proof}

\begin{remark}
	When $z = 0$, we have $\G (t, 1) = \G (0,1) = 1$ for all $t \ge 0$. This can be seen from \eqref{eq:PDE} which shows $\partial _t \G (t, 1) = 0$, so the total mass is conserved. We can similarly derive conservation of the first moment by $\partial _t \partial _z \G (t, 1) = -\partial _z \G (t, 1) + \mu$ from \eqref{eq:PDE}.
\end{remark}

\subsection{Convergence of the auxiliary function}
\label{subsec:4.2}

In this subsection, we first show that the auxiliary $v$ function satisfies an integral equation, and then prove it converges to a limit which depends on the value of $\mu$.

\begin{lemma}
	\label{lem:integral}
	For $t \ge 0$, $v (t)$ satisfies
	\begin{align*}
		v (t) = 1 - f _0 (t) + f _0 (0) \expo ^{-\int _0 ^{t} \log v (s) \dd s} + \mu \int _0 ^{t} [v (s)] ^{\expo ^{-t + s}} \expo ^{-t + s} \dd s,
	\end{align*}
	where $f _0 (t) \coloneqq \G (0, 1 - \expo ^{-t})$.
\end{lemma}

\begin{proof}
	We recall that $\G (t, 0) = p _0 (t)$. Applying Lemma \ref{lem:solution-to-pde} with $z = 1$ we obtain
	\begin{align}
		\notag
		&\G (t, 0)
		= 1 + \left(
			\G (0, 1 - \expo ^{-t}) - 1 - \mu \int _0 ^t [v (s)] ^{\expo ^{-t + s}} \expo ^{-t + s} \dd s
		\right) [v (t)] ^{-1} \\
		\label{eq:vt1}
		\implies & v (t) = v (t) p _0 (t) + 1 - f _0 (t) + \mu \int _0 ^t [v (s)] ^{\expo ^{-t + s}} \expo ^{-t + s} \dd s.
	\end{align}
	It remains to solve for $p _0 (t)$ in terms of $v$. Using the differential equation satisfied by $p _0 (t)$, we know that
	\begin{align}
		\label{eq:p0}
		p _0 (t) \exp \left(\int _0 ^t \mu - p _1 (s) \dd s \right) = p _0 (0) = \G (0, 0) = f _0 (0).
	\end{align}
	In view of \eqref{eq:H}, we can solve $p _0 (t)$ in terms of $v$ by
	\begin{align*}
		& a (s) = \expo ^{-s} H' (s) \\
		\implies & \int _0 ^t a (s) \dd s = \expo ^{-t} H (t) + \int _0 ^t H (s) \expo ^{-s} \dd s = \log v (t) + \int _0 ^t \log v (s) \dd s \\
		\implies & p _0 (t) = f _0 (0) \exp \left(-\int _0 ^t a (s) \dd s\right) = f _0 (0) [v (t)] ^{-1} \expo ^{-\int _0 ^t \log v (s) \dd s}.
	\end{align*}
	Combine with \eqref{eq:vt1} we conclude the proof.
\end{proof}

We now investigate the limiting behavior of $v (t)$ as $t \to \infty$. First, since $0 \le p _1 (s) \le 1$ for all $s \in [0, t]$, we can bound $v$ using its definition \eqref{eq:auxillary}:
\begin{align*}
	(\mu - 1) \int _0 ^t \expo ^{-t + s} \dd s \le \log v (t) \le \mu \int _0 ^t \expo ^{-t + s} \dd s.
\end{align*}
By direct computations, we obtain the following bound:
\begin{align}
	\label{eq:v-bound}
	(\mu - 1) (1 - \expo ^{-t}) \le \log v (t) \le \mu (1 - \expo ^{-t}).
\end{align}
In particular, $1 \le v (t) \le \expo ^\mu$.

Next, we control the nonlinear integral term using the following lemma.
\begin{lemma}
	\label{lem:squeeze}
	Let $t _0 \ge 0$. If $1 \le m \le v (t) \le M$ for all $t \ge t _0$, then
	\begin{align*}
		\varphi (m) - e _1 (t) \le \mu \int _0 ^t [v (s)] ^{\expo ^{-t + s}} \expo ^{-t + s} \dd s \le \varphi (M) + e _1 (t), \qquad \forall t \ge t _0,
	\end{align*}
	where $\varphi: (0, +\infty) \to \mathbb R _+$ is a strictly increasing continuous function defined by
	\begin{align*}
		\varphi (x) \coloneqq \mu \int _0 ^\infty x ^{\expo ^{-t}} \expo ^{-t} \dd t = \begin{cases}
				\mu \cdot \dfrac{x - 1}{\log x}, & x > 0, x \neq 1 \\
				\mu, & x = 1
		\end{cases}
	\end{align*}
	and the remainder term is $e _1 (t) \coloneqq \mu \expo^{\mu + t _0} \expo ^{-t}$.
\end{lemma}

\begin{proof}
	First, we separate our target integral as follows:
	\begin{align*}
		\mu \int _0 ^t [v (s)] ^{\expo ^{-t + s}} \expo ^{-t + s} \dd s = \mu \int _0 ^{t _0} [v (s)] ^{\expo ^{-t + s}} \expo ^{-t + s} \dd s + \mu \int _{t _0} ^t [v (s)] ^{\expo ^{-t + s}} \expo ^{-t + s} \dd s.
	\end{align*}
	The first term has an exponential decay since
	\begin{align*}
		\mu \int _0 ^{t _0} [v (s)] ^{\expo ^{-t + s}} \expo ^{-t + s} \dd s \le \mu \int _0 ^{t _0} \expo ^{\mu \expo ^{-t + s}} \expo ^{-t + s} \dd s \le \mu \expo ^\mu \expo ^{-t + t _0}.
	\end{align*}
	The second term is bounded from above by
	\begin{align*}
		\mu \int _{t _0} ^t [v (s)] ^{\expo ^{-t + s}} \expo ^{-t + s} \dd s \le \mu \int _{t _0} ^t M ^{\expo ^{-t + s}} \expo ^{-t + s} \dd s &= \mu \int _{0} ^{t - t _0} M ^{\expo ^{-s}} \expo ^{-s} \dd s \le \varphi (M)
	\end{align*}
	and from below by
	\begin{align*}
		\mu \int _{t _0} ^t [v (s)] ^{\expo ^{-t + s}} \expo ^{-t + s} \dd s &\ge \mu \int _{t _0} ^t m ^{\expo ^{-t + s}} \expo ^{-t + s} \dd s = \mu \int _{0} ^{t - t _0} m ^{\expo ^{-s}} \expo ^{-s} \dd s \\
		&= \varphi (m) - \mu \int _{t - t _0} ^\infty m ^{\expo ^{-s}} \expo ^{-s} \dd s \ge \varphi (m) - \mu m \expo ^{-t + t _0}.
	\end{align*}
	The proof is thus completed because $m \le v (t) \le \expo ^\mu$ and $\mu m \expo ^{-t + t _0} \le e _1 (t)$.
\end{proof}

We are now ready to prove the convergence of $v (t)$.

\begin{lemma}
	\label{lem:convergence-v}
	Let $\nu = \mu + W _0 (-\mu \expo ^{-\mu})$, then
	\begin{align*}
		\lim _{t \to +\infty} v (t) = \expo ^\nu.
	\end{align*}
\end{lemma}

\begin{proof}
	Denote
	\begin{align}
		\label{def:e2}
		e _2 (t) = 1 - f _0 (t) + f _0 (0) \exp \left(-\int _0 ^{t} \log v (s) \dd s \right).
	\end{align}
	Since
	\begin{align*}
		&0 \le 1 - f _0 (t) = 1 - \G (0, 1 - \expo ^{-t}) \le \partial _z \G (0, 1) \expo ^{-t} = \mu \expo ^{-t}, \\
		&\int _0 ^{t} \log v (s) \dd s \ge \int _0 ^{t} (\mu - 1) (1 - \expo ^{-s}) \dd s = (\mu - 1) (t + \expo ^{-t} - 1) \ge (\mu - 1) (t - 1),
	\end{align*}
	we deduce that
	\begin{align}
		\label{eq:e2}
		e _2 (t) \le \mu \expo ^{-t} + f _0 (0) \expo ^{-(\mu - 1) (t - 1)} \to 0 \qquad \text{ as } t \to +\infty.
	\end{align}
	With the above error function, we can write the integral equation in Lemma \eqref{lem:integral} as
	\begin{align}
		\label{eq:v-error}
		v (t) = e _2 (t) + \mu \int _0 ^{t} [v (s)] ^{\expo ^{-t + s}} \expo ^{-t + s} \dd s.
	\end{align}
	Since $v (t) \in [1, \expo^\mu]$ is bounded (uniformly in time), $e _2(t) \to 0$ as $t \to +\infty$, the liminf and limsup of $v$ exist and are bounded between
	\begin{align*}
		m &\coloneqq \liminf _{t \to +\infty} v (t) \ge \lim _{t \to +\infty} \exp \left((\mu - 1) (1 - \expo ^{-t})\right) = \expo ^{\mu - 1}, \\
		M &\coloneqq \limsup _{t \to +\infty} v (t) \le \lim _{t \to +\infty} \exp \left(\mu (1 - \expo ^{-t})\right) = \expo ^\mu.
	\end{align*}

	We claim that $\varphi (m) \le m \le M \le \varphi (M)$.
	To prove this, we first fix $\varepsilon > 0$ with $\varepsilon < \expo ^{\mu - 1} - 1$, then there exists $t _0 > 0$ such that
	\begin{align*}
		1 < m - \varepsilon \le v (t) \le M + \varepsilon, \qquad \forall t \ge t _0.
	\end{align*}
	Invoking Lemma \ref{lem:squeeze} together with the monotonicity of the map $\varphi$, we conclude for all $t \ge t _0$ that
	\begin{align*}
		\varphi (m - \varepsilon) - e _1 (t) \le \mu \int _0 ^t [v (s)] ^{\expo ^{-t + s}} \expo ^{-t + s} \dd s \le \varphi (M + \varepsilon) + e _1 (t), \qquad \forall t \ge t _0.
	\end{align*}
	Therefore, the limsup and liminf of $v$ are bounded by
	\begin{align*}
		M &= \limsup _{t \to +\infty} v (t) \le \lim _{t \to +\infty} e _2 (t) + \varphi (M + \varepsilon) + \lim _{t \to +\infty} e _1 (t) = \varphi (M + \varepsilon), \\
		m &= \liminf _{t \to +\infty} v (t) \ge \lim _{t \to +\infty} e _2 (t) + \varphi (m - \varepsilon) - \lim _{t \to +\infty} e _1 (t) = \varphi (m - \varepsilon).
	\end{align*}
	This is true for any sufficiently small $\varepsilon$ so the advertised claim $\varphi (m) \le m \le M \le \varphi (M)$ is justified.

	Within the interval $[\expo ^{\mu - 1}, \expo ^\mu]$, we demonstrate that the function $\varphi$ is a contraction mapping. Indeed, note that
	\begin{align*}
		\varphi (\expo ^{\mu - 1}) &= \mu \cdot \frac{\expo ^{\mu - 1} - 1}{\mu - 1} = \expo ^{\mu - 1} + \frac{\expo ^{\mu - 1} - \mu}{\mu - 1} > \expo ^{\mu - 1}, \\
		\varphi (\expo ^\mu) &= \mu \cdot \frac{\expo ^\mu - 1}{\mu} = \expo ^\mu - 1 < \expo ^\mu.
	\end{align*}
	Moreover, the derivatives of $\varphi$ are
	\begin{align*}
		\varphi' (\expo ^x) &= \mu x ^{-2} \expo ^{-x} (1 - \expo ^{x} + x \expo ^{x}) > 0, \\
		\varphi'' (\expo ^x) &= \mu x ^{-3} \expo ^{-x} (2 - 2 \expo ^{-x} - x \expo ^{-x} - x) < 0.
	\end{align*}
	So for $x \in [\mu - 1, \mu]$, we know $\varphi' (\expo ^x)$ is bounded by
	\begin{align*}
		0 < \varphi' (\expo ^\mu) \le \varphi' (\expo ^x) \le \varphi' (\expo ^{\mu - 1}) = \frac{\mu ^2 - 2 \mu + \mu \expo ^{1 - \mu}}{\mu ^2 - 2 \mu + 1} \coloneqq L _\mu < 1.
	\end{align*}

	Since $\varphi: [\expo ^{\mu - 1}, \expo ^{\mu}] \to [\expo ^{\mu - 1}, \expo ^{\mu}]$ is a contraction mapping and $[m, M]$ is contained in this interval, we conclude from $\varphi (m) \le m \le M \le \varphi (M)$ that $m = M = \expo ^{\nu}$ where $\expo ^\nu$ is the unique fixed point of $\varphi$ in $[\expo ^{\mu - 1}, \expo ^{\mu}]$. It satisfies
	\begin{align*}
		\expo ^\nu = \mu \frac{\expo ^\nu - 1}{\nu} &\iff \nu \expo ^{\nu} = \mu \expo ^{\nu} - \mu
		\iff (\nu - \mu) \expo ^{\nu - \mu} = -\mu \expo ^{-\mu}.
	\end{align*}
	 We observe that $\nu - \mu$ and $-\mu$ are two real roots to the equation $z \expo ^z = -\mu \expo ^{-\mu}$, hence we can use the Lambert $W$ function to select the principal branch and arrive at the relation $\nu - \mu = W _0 (-\mu \expo ^{-\mu})$. This finishes the proof of Lemma \ref{lem:convergence-v}.
\end{proof}

\begin{remark}
	The convergence of $v$ implies the pointwise convergence of the probability generating function. Indeed, sending $t \to +\infty$ in Lemma \ref{lem:solution-to-pde} and applying the result of Lemma \ref{lem:squeeze} (with $v$ being replaced by $v^z$), we obtain
	\begin{align*}
		\lim _{t \to +\infty} \G (t, 1 - z) &= 1 + \left(
			\G (0, 1) - 1 - z \varphi (\expo ^{\nu z})
		\right) \expo ^{-\nu z}
		\\
		&= 1 - \mu \cdot \frac{1 - \expo ^{-\nu z}}{\nu} = \frac{\mu - \nu}{\nu} \left(\expo ^{\nu (1 - z)} - 1\right) = G _{\bpb} (1 - z),
	\end{align*}
	where $G _{\bpb}$ denotes the PGF of the zero-truncated Poisson distribution $\bpb$ \eqref{eq:zero_truncated_poisson}. According to a classical result \cite{esquivel_probability_2008}, if the PGF exists and the above convergence holds in a neighborhood of $z = 0$ then $\bp (t)$ converges to $\bpb$ in distribution. However, in the following we will strengthen the results obtained so far and prove $\ell^p$ convergence for $p=1,2$, which are stronger convergence guarantees.
\end{remark}

\subsection{Quantitative convergence of the auxiliary function}
\label{subsec:4.3}

Next, we study the quantitative rate of convergence of $v (t) \xrightarrow{t \to \infty} \expo ^{\nu}$. We do it in two steps. In the first step we prove a $\cO (\expo ^{- c \sqrt t})$ decay, and in the second step we refine the previous estimate to reach a $\cO (\expo ^{-c t})$ decay. First, we show an improved lower bound.

\begin{lemma}
	\label{lem:lower-bound}
	There exists $T _0 = T _0 (\mu) > 0$ such that $v (t) \ge \expo ^{\mu - 1}$ for all $t \ge T _0$.
\end{lemma}

\begin{proof}
	For some $t _0 > 0$ to be determined, since $v (t) \ge \expo ^{(\mu - 1) (1 - \expo ^{-t})} \ge \expo ^{(\mu - 1) (1 - \expo ^{-t _0})}$ for all $t \ge t _0$, by Lemma \ref{lem:squeeze} we know that
	\begin{align*}
		v (t) \ge e _2 (t) + \varphi \left(\expo ^{(\mu - 1) (1 - \expo ^{-t _0})}\right) - e _1 (t) \ge \varphi \left(\expo ^{(\mu - 1)(1 - \expo ^{-t _0})}\right) - \mu \expo ^{\mu + t _0} \expo ^{-t}
	\end{align*}
	holds for all $t \ge t _0$. In particular, for all $t \ge 2 t _0$, we obtain
	\begin{align*}
		v (t) \ge \varphi\left(\expo ^{(\mu - 1)(1 - \expo ^{-t _0})}\right) - \mu \expo ^{\mu - t _0} \to \varphi (\expo ^{\mu - 1}) \qquad \text{ as } t _0 \to +\infty.
	\end{align*}
	Since $\varphi (\expo ^{\mu - 1}) > \expo ^{\mu - 1}$ (recall the proof of Lemma \ref{lem:convergence-v}), we can find $T _0 \coloneqq 2 t _0$ for some sufficiently large $t _0$ depending on $\mu$ such that $v (t) \ge \expo ^{\mu - 1}$ for all $t \ge T _0$.
\end{proof}

\begin{lemma}
	\label{lem:crude}
	There exist constants $C, \delta > 0$ depending only on $\mu$, such that
	\begin{align*}
		\left\lvert v (t) - \expo ^\nu \right\rvert \le C \expo ^{- \sqrt {\delta t}}, \qquad \forall t > 0.
	\end{align*}
\end{lemma}

\begin{proof}
	For $t \ge 0$, we set
	\begin{align*}
		r (t) \coloneqq \sup _{s \ge t} \left\lvert v (s) - \expo ^\nu \right\rvert,
	\end{align*}
	which is a nonnegative and decreasing function.

	Let $t _k \ge T _0$ be a sequence of increasing times to be determined later. Denote $m _k = \expo ^\nu - r (t _k)$ and $M _k = \expo ^\nu + r (t _k)$, then $m _k \le v (t) \le M _k$ for all $t \ge t _k$. Using Lemma \ref{lem:lower-bound}, we deduce that $\expo ^{\mu - 1} \le m _k \le M _k \le \expo ^{\mu}$. By Lemma \ref{lem:squeeze}, we know that at any time $s \ge t _k$, it holds that
	\begin{align*}
		\varphi (m _k) - \mu \expo ^{\mu + t _k - s} \le v (s) - e _2 (s) \le \varphi (M _k) + \mu \expo ^{\mu + t _k - s}.
	\end{align*}
	In particular,
	\begin{align*}
		\left\lvert v (s) - \expo ^\nu \right\rvert \le \max \{\expo ^\nu - \varphi (m _k), \varphi (M _k) - \expo ^\nu\} + e _2 (s) + \mu \expo ^{\mu + t _k - s}.
	\end{align*}
	Since $\varphi$ is $L _\mu$-Lipschitz in $[\expo ^{\mu - 1}, \expo ^\mu]$, we have
	\begin{align*}
		\max \{\expo ^\nu - \varphi (m _k), \varphi (M _k) - \expo ^\nu\} \le L _\mu r (t _k).
	\end{align*}
	Now we take the supremum over $s \ge t _{k + 1}$ to obtain the following recursive inequality
	\begin{align*}
		r (t _{k + 1}) &\le L _\mu r (t _k) + \sup _{s \ge t _{k + 1}} e _2 (s) + \mu \expo ^{\mu + t _k - t _{k + 1}} \\
		&\le L _\mu r (t _k) + C \expo ^{t _k - t _{k + 1}} + \expo ^{-(\mu - 1) (t _{k + 1} - 1)},
	\end{align*}
	in which $C > 0$ depends only on $\mu$ and we employed the bound \eqref{eq:e2} for $e _2 (s)$.

	Denote $\delta = - \log L _\mu > 0$. For $k \ge 0$, we take $t _k \coloneqq T _0 + k ^2 \delta$, then
	\begin{align*}
		& \expo ^{t _k - t _{k + 1}} = \expo ^{-(2 k + 1) \delta} = L _\mu ^{2 k + 1}, \\
		& \expo ^{-(\mu - 1) (t _{k + 1} - 1)} \le \expo ^{-(\mu - 1) [(k + 1) ^2 \delta - 1]} \le C\expo ^{-(2 k + 1) \delta} = C L _\mu ^{2 k + 1},
	\end{align*}
	where $C > 0$ is chosen such that $\log C + (\mu - 1) [(k + 1)^2 \delta - 1] \ge (2 k + 1) \delta$ for all $k \ge 0$. The recursive inequality now becomes
	\begin{align*}
		r (t _{k + 1}) &\le L _\mu r (t _k) + C L _\mu ^{2 k + 1}, \\
		L _\mu ^{-k - 1} r (t _{k + 1}) &\le L _\mu ^{-k} r (t _k) + C L _\mu ^k.
	\end{align*}
	Taking the summation from $0$ to $k - 1$, we have
	\begin{align*}
		L _\mu ^{-k} r (t _k) &\le r (t _0) + \frac C{1 - L_\mu} = C, \\
		r (t _k) &\le C L _\mu ^k = C \expo ^{-\sqrt{\delta (t _k - T _0)}} = C \expo ^{-\sqrt {\delta t _k}}.
	\end{align*}
	We thus conclude by monotonicity that
	\begin{align*}
		r (t) \le C \expo ^{-\sqrt {\delta t}}, \qquad t \ge T _0.
	\end{align*}
	Finally, the restriction $t \ge T _0$ can be easily removed by taking $C$ to be sufficiently large.
\end{proof}

We can see that both the convergence and the above estimate are based on comparison. To obtain a sharper estimate, we take the difference:
\begin{align}
	\notag
	v (t) - \expo ^\nu &= e _2 (t) + \mu \int _0 ^t [v (s)] ^{\expo ^{-t + s}} \expo ^{-t + s} \dd s - \mu \int _0 ^\infty \expo ^{\nu \expo ^{-s}} \expo ^{-s} \dd s \\
	\notag
	&= e _2 (t) + \mu \int _0 ^t [v (s)] ^{\expo ^{-t + s}} \expo ^{-t + s} \dd s \\
	\notag
	& \qquad - \mu \int _0 ^t \expo ^{\nu \expo ^{-t + s}} \expo ^{-t + s} \dd s - \mu \int _t ^\infty \expo ^{\nu \expo ^{-s}} \expo ^{-s} \dd s \\
	\label{eq:diff-v-0}
	&= e _2 (t) - \expo ^{-t} \varphi (\expo ^{\nu \expo ^{-t}}) + \mu \int _0 ^t \left(
		[v (s)] ^{\expo ^{-t + s}} - \expo ^{\nu \expo ^{-t + s}}
	\right) \expo ^{-t + s} \dd s.
\end{align}
Therefore, we can control the difference by
\begin{align}
	\label{eq:diff-v}
	\left\lvert v (t) - \expo ^\nu \right\rvert &\le \left\lvert e _2 (t) - \expo ^{-t} \varphi (\expo ^{\nu \expo ^{-t}}) \right\rvert + \mu \int _0 ^t \left\lvert
		[v (s)] ^{\expo ^{-t + s}} - \expo ^{\nu \expo ^{-t + s}}
	\right\rvert \expo ^{-t + s} \dd s \\
	\label{eq:diff-v-2}
	&\le \left\lvert e _2 (t) - \expo ^{-t} \varphi (\expo ^{\nu \expo ^{-t}}) \right\rvert + \mu \int _0 ^t \left\lvert
		v (s) - \expo ^\nu
	\right\rvert \expo ^{-2t + 2s} \dd s.
\end{align}
Here we used the fact that the power function $x \mapsto x ^\alpha$ is $\alpha$-Lipschitz on $[1, \infty)$ for $0 < \alpha < 1$.
Define this integral quantity on the right by
\begin{align*}
	y _2 (t) \coloneqq \int _0 ^t \left\lvert
		v (s) - \expo ^\nu
	\right\rvert \expo ^{-2t + 2s} \dd s.
\end{align*}
Then $y _2 ' (t) = \left\lvert v (t) - \expo ^\nu \right\rvert - 2 y _2 (t)$, so $y _2$ satisfies the following differential inequality
\begin{align*}
	y _2 ' (t) + (2 - \mu) y _2 (t) \le \left\lvert e _2 (t) - \expo ^{-t} \varphi (\expo ^{\nu \expo ^{-t}}) \right\rvert.
\end{align*}
If $\mu < 2$, then $y _2$ will decay exponentially. However, this estimate is useless when $\mu > 2$. We need to estimate the difference in the following integral sense.

\begin{lemma}
	\label{lem:finer}
	Denote $\bar c = \nu \wedge 1$. For any $c < \bar c$, define
	\begin{align*}
		y _{c} (t) &= \int _t ^\infty \left\lvert v (s) - \expo ^\nu \right\rvert \expo ^{-c t + c s} \dd s.
	\end{align*}
	If $y _{c} (0) < \infty$ converges, then for $c' = c + \frac{\expo ^{-\nu}}2 (1 - c)$, it holds for all $t \ge 0$ that
	\begin{align}
		\label{eq:yc_exponential decay}
		y _c (t) + y _2 (t) \le C y _c (0) \expo ^{-c' t} +
		\begin{cases}
			\frac C{(\bar c - c) (\bar c - c')} \expo ^{-c' t }, & c' < \bar c, \\
			\frac {C}{\bar c - c} t \expo ^{-\bar c t}, & c' = \bar c, \\
			\frac C{(\bar c - c) (c' - \bar c)} \expo ^{-\bar c t}, & c' > \bar c,
		\end{cases}
	\end{align}
	in which $C > 0$ is a constant depending only on $\mu$.
\end{lemma}

\begin{proof}
	Notice that
	\begin{align*}
		y _c' (t) = - \left\lvert v (t) - \expo ^\nu \right\rvert - c y _c (t), \qquad y _2 ' (t) = \left\lvert v (t) - \expo ^\nu \right\rvert - 2 y _2 (t).
	\end{align*}
	Combined they satisfy the following differential equality:
	\begin{align}
		\label{eq:diff-ineq}
		(y _c + y _2)' (t) &= - c y _c (t) - 2 y _2 (t) = - c (y _c (t) + y _2 (t)) - (2 - c) y _2 (t).
	\end{align}
	As $c < 2$, if we drop the last term then we directly get a decay at rate
	\begin{align}
		\label{eq:e-ct-decay}
		y _c (t) + y _2 (t) \le y _c (0) \expo ^{-c t}, \qquad \forall t \ge 0.
	\end{align}
	In the following, we would like to improve the decay rate from $c$ to $c'$.

	By \eqref{eq:diff-v}, we can estimate $y _c$ by
	\begin{align*}
		y _c (t) \le E _c (t) + \mu \int _t ^\infty \int _0 ^s
		\left\lvert
			[v (r)] ^{\expo ^{-s + r}} \expo ^{-s + r}  - \expo ^{\nu \expo ^{-s + r}} \expo ^{-s + r}
		\right\rvert \expo ^{-ct + cs}\dd r
		\dd s,
	\end{align*}
	where
	\begin{align*}
		E _c (t) \coloneqq \int _t ^\infty \left\lvert e _2 (s) - \expo ^{-s} \varphi (\expo ^{\nu \expo ^{-s}}) \right\rvert \expo ^{-ct + cs} \dd s.
	\end{align*}
	Note that now we have an improved estimate of $e _2$. Indeed, using the crude estimate established in Lemma \ref{lem:crude}, we have
	\begin{align*}
		\int _0 ^t \left\lvert \log v (s) - \nu \right\rvert \dd s \le \int _0 ^t \left\lvert v (s) - \expo ^\nu \right\rvert \dd s \le C.
	\end{align*}
	Therefore we can improve the estimate \eqref{eq:e2} to
	\begin{align}
		\label{eq:e2_improved}
		e _2 (t) \le \mu \expo ^{-t} + C \expo ^{-\nu t} \le C \expo ^{-\bar c t}.
	\end{align}
	As $\varphi (\expo ^{\nu \expo ^{-s}}) \le \varphi (\expo ^\nu)$ is bounded, we have
	\begin{align*}
		E _c (t) \le C \int _t ^\infty \expo ^{-\bar c s} \expo ^{-c t + c s} \dd s = \frac C{\bar c - c} \expo ^{-\bar c t}, \qquad \forall t \ge 0.
	\end{align*}
	We now exchange the order of the double integral:
	\begin{align*}
		&\mu \int _t ^\infty \int _0 ^s
		\left\lvert
			[v (r)] ^{\expo ^{-s + r}} \expo ^{-s + r}  - \expo ^{\nu \expo ^{-s + r}} \expo ^{-s + r}
		\right\rvert \expo ^{-ct + cs} \dd r
		\dd s \\
		&\qquad = \mu \int _0 ^\infty \int _{r \vee t} ^\infty
		\left\lvert
			[v (r)] ^{\expo ^{-s + r}} \expo ^{-s + r}  - \expo ^{\nu \expo ^{-s + r}} \expo ^{-s + r}
		\right\rvert \expo ^{-cr + cs} \dd s \expo ^{-ct + cr}
		\dd r \\
		&\qquad = \mu \int _0 ^\infty \left\lvert
			\int _{r \vee t} ^\infty
			[v (r)] ^{\expo ^{-s + r}} \expo ^{-(1 - c) (s - r)}  - \expo ^{\nu \expo ^{-s + r}} \expo ^{-(1 - c) (s - r)} \dd s
		\right\rvert \expo ^{-ct + cr} \dd r \\
		&\qquad = \mu \int _0 ^\infty \left\lvert
			\int _{(t - r) _+} ^\infty
			[v (r)] ^{\expo ^{-s}} \expo ^{- (1 - c) s}  - \expo ^{\nu \expo ^{-s}} \expo ^{- (1 - c)s} \dd s
		\right\rvert \expo ^{-ct + cr} \dd r.
	\end{align*}
	The absolute value symbol is extracted because the (inner) integrand has a fixed sign for each $r$ and $t$. To compute the inner integral, we define
	\begin{align*}
		\varphi _c (x) \coloneqq \mu \int _0 ^\infty x ^{\expo ^{-s}} \expo ^{- (1 - c) s} \dd s.
	\end{align*}
	Then the inner integral can be expressed as
	\begin{align*}
		&\mu
			\int _{(t - r) _+} ^\infty
			[v (r)] ^{\expo ^{-s}} \expo ^{-(1 - c) s}  - \expo ^{\nu \expo ^{-s}} \expo ^{-(1 - c) s} \dd s
		\\
		&\qquad = \begin{cases}
			\expo ^{- (1 - c) (t - r)}
			\left(
				\varphi _c ([v (r)] ^{\expo ^{-t + r}}) - \varphi _c (\expo ^{\nu \expo ^{-t + r}})
			\right), & r < t, \\
			\varphi _c (v (r)) - \varphi _c (\expo ^\nu), & r \ge t.
		\end{cases}
	\end{align*}
	We will leave the detailed proof of the Lipschitzness of $\varphi _c$ in Lemma \ref{lem:lipschitz-phic}. Denote $\epsilon _c = \frac{\mu \expo ^{-\nu}}{2(2 - c)} (1 - c) \in (0,1)$, then we show in Lemma \ref{lem:lipschitz-phic} that $\varphi _c$ is $(\mu - \epsilon _c)$-Lipschitz on $[1, \infty)$, and $(1 - \epsilon _c)$-Lipschitz on $[\expo ^\nu - \epsilon _c / \mu, +\infty)$. By Lemma \ref{lem:crude}, for $t > T _c$ with
	\begin{align*}
		T _c = \frac1\delta \log ^2 \left(\frac {C \mu}{\epsilon _c}\right) \le C \log ^2 \left(\frac{2}{1 - c}\right),
	\end{align*}
	we have $v (t) \in [\expo ^\nu - \epsilon _c / \mu, \expo ^\nu + \epsilon _c /\mu]$, so
	\begin{align*}
		&\mu \left\lvert
			\int _{(t - r) _+} ^\infty
			[v (r)] ^{\expo ^{-s}} \expo ^{-(1 - c) s}  - \expo ^{\nu \expo ^{-s}} \expo ^{-(1 - c) s} \dd s
		\right\rvert \\
		&\qquad \le \begin{cases}
			(\mu - \epsilon _c) \left\lvert v (r) - \expo ^\nu \right\rvert \expo ^{- (2 - c) (t - r)}, & r < t, \\
			(1 - \epsilon _c) \left\lvert v (r) - \expo ^\nu \right\rvert, & r \ge t.
		\end{cases}
	\end{align*}
	Thus
	\begin{align}
		\notag
		y _c (t) &\le E _c (t) + (\mu - \epsilon _c) \int _0 ^t \left\lvert v (r) - \expo ^\nu \right\rvert \expo ^{-2 t + 2 r} \dd r \\
		\notag
		&\qquad + (1 - \epsilon _c) \int _t ^\infty \left\lvert v (r) - \expo ^\nu \right\rvert \expo ^{-c t + c r} \dd r \\
		\label{eqn:yct}
		&= E _c (t) + (\mu - \epsilon _c) y _2 (t) + (1 - \epsilon _c) y _c (t).
	\end{align}
	Rearranging \eqref{eqn:yct} leads us to
	\begin{align*}
		& \epsilon _c (y _c + y _2) (t) \le E _c (t) + \mu y _2 (t) \\
		\implies & \frac{\epsilon _c (2 - c)}{\mu} (y _c + y _2) (t) \le \frac{2 - c}{\mu}  E _c (t) + (2 - c) y _2 (t).
	\end{align*}
	Combining it with \eqref{eq:diff-ineq} we obtain
	\begin{align*}
		(y _c + y _2)' (t) + \left(c + \frac{\expo ^{-\nu}}{2} (1 - c) \right) (y _c + y _2) (t) \le 2 E _c (t).
	\end{align*}
	Recall that $c' = c + \frac12 (1 - c) \expo ^{-\nu}$. We have
	\begin{align*}
		\frac{\dd}{\dd t} \left[(y _c + y _2) (t) \expo ^{c' t}\right] \le 2 E _c (t) \expo ^{c' t}.
	\end{align*}
	Upon integration this inequality from $T _c$ to $t$, we get for all $t \ge T _c$ that
	\begin{align*}
		(y _c + y _2) (t) & \le (y _c + y _2) (T _c) \expo ^{-c' (t - T _c)} + 2 \int _{T _c} ^t E _c (s) \expo ^{-c' (t - s)} \dd s.
	\end{align*}
	Note that \eqref{eq:e-ct-decay} yields
	\begin{align*}
		(y _c + y _2) (T _c) \expo ^{-c' (t - T _c)} \le y _c (0) \expo ^{-c T _c} \expo ^{-c' (t - T _c)} \le \expo^{(c' - c) T _c} y _c (0) \expo ^{-c't}.
	\end{align*}
	The right hand side also dominates $(y _c + y _2) (t)$ for $t \le T _c$. Indeed, we have from \eqref{eq:e-ct-decay} that
	\begin{align*}
		(y _c + y _2) (t) \le y _c (0) \expo ^{ct} = y_c (0) \expo^{-c T _c} \expo^{c (T _c - t)}
		& \le y _c (0) \expo ^{-c T _c} \expo^{c' (T _c -t)} \\
		&= y _c (0) \expo ^{-c' t} \expo ^{(c' - c) T _c}.
	\end{align*}
	Next, notice that for $c \in (0, 1)$, $(c' - c) T _c$ is uniformly bounded by
	\begin{align*}
		(c' - c) T _c \le C \log ^2 \left(\frac2{1 - c}\right) \cdot \frac{\expo ^{-\nu}}2 (1 - c) \le C.
	\end{align*}
	As for the error term, we have
	\begin{align*}
		\int _{T _c} ^t E _c (s) \expo ^{-c' (t - s)} \dd s &\le \int _{T _c} ^t \frac C{\bar c - c} \expo ^{-\bar c s} \expo ^{-c' t + c' s} \dd s \\
		&\le
		\begin{cases}
			\frac C{(\bar c - c) (\bar c - c')} \expo ^{-c' t - (\bar c - c') T _c}, & c' < \bar c, \\
			\frac {C}{\bar c - c} (t - T _c) \expo ^{-\bar c t}, & c' = \bar c, \\
			\frac C{(\bar c - c) (c' - \bar c)} \expo ^{-\bar c t}, & c' > \bar c,
		\end{cases}
	\end{align*}
	from which the advertised estimate \eqref{eq:yc_exponential decay} follows. 
\end{proof}

\begin{lemma}[Lipschitzness of the function $\varphi_c$]
	\label{lem:lipschitz-phic}
	Denote $\epsilon _c = \frac{\mu \expo ^{-\nu}}{2 (2 - c)} (1 - c) \in (0,1)$, then $\varphi _c$ is $(\mu - \epsilon _c)$-Lipschitz on $[1, \infty)$, and $(1 - \epsilon _c)$-Lipschitz on $[\expo ^\nu - \epsilon _c / \mu, +\infty)$.
\end{lemma}

\begin{proof}
	We remark that $\varphi _c$ can be computed and expressed using Euler's Gamma function, but we will not need the exact form.
	Its derivative
	\begin{align*}
		\varphi _c ' (x) = \mu \int _0 ^\infty \expo ^{-s} x ^{\expo ^{-s} - 1} \expo ^{- (1 - c) s} \dd s = \mu \int _0 ^\infty x ^{\expo ^{-s} - 1} \expo ^{- (2 - c) s} \dd s
	\end{align*}
	is positive and decreasing in $[1, \infty)$. Thus for all $x\ge 1$,
	\begin{align*}
		\varphi _c ' (x) \le \varphi _c ' (1) = \mu \int _0 ^\infty \expo ^{- (2 - c) s} \dd s = \frac{\mu}{2 - c} \le \mu - \epsilon _c.
	\end{align*}
	Moreover, the partial derivative of $\varphi' _c$ with respect to $c$ reads as
	\begin{align*}
		&\partial _c \varphi ' _c (x) = \mu \int _0 ^\infty x ^{\expo ^{-s} - 1} \expo ^{- (2 - c) s} s \dd s \ge \frac\mu x \int _0 ^\infty s \expo ^{- (2 - c) s} \dd s = \frac{1}{(2 - c) ^2} \frac{\mu}{x},
	\end{align*}
	hence for any $c \le 1$ it holds
	\begin{align*}
		\varphi' _c (x) \le \varphi' _1 (x) - \frac\mu x \int _{c} ^1 \frac1{(2 - t) ^2} \dd t = \frac{\varphi (x)}x - \frac{1 - c}{2 - c} \cdot \frac\mu x.
	\end{align*}
	In particular,
	\begin{align*}
		\varphi _c ' (\expo ^\nu) \le 1 - \frac{1 - c}{2 - c} \cdot \frac\mu{\expo ^\nu} = 1 - 2 \epsilon _c < 1.
	\end{align*}
	On the other hand, for $x \ge 1$ we have
	\begin{align*}
		\varphi _c '' (x) = -\mu \int _0 ^\infty \expo ^{-s} (1 - \expo ^{-s}) x ^{\expo ^{-s} - 2} \expo ^{- (1 - c) s} \dd s &\ge -\mu \int _0 ^\infty \expo ^{-s} (1 - \expo ^{-s}) \expo ^{- (1 - c) s} \dd s \\
		&= -\frac\mu{(2 - c) (3 - c)} \ge -\mu.
	\end{align*}
	Therefore, if $x \in [\expo ^\nu - \epsilon _c / \mu, \expo ^\nu]$, then
	\begin{align*}
		\varphi _c ' (x) \le \varphi _c ' (\expo ^\nu) + \mu (\expo ^\nu - x) \le 1 - 2 \epsilon _c + \epsilon _c = 1 - \epsilon _c.
	\end{align*}
	If $x \ge \expo ^\nu$, then $\varphi _c ' (x) \le \varphi _c ' (\expo ^\nu) \le 1 - 2 \epsilon _c$. We thus conclude that the function $\varphi _c$ is $(1 - \epsilon _c)$-Lipschitz on $[\expo ^\nu - \epsilon _c / \mu, +\infty)$.
\end{proof}

Lemma \ref{lem:finer} provides an iteration scheme to improve the decay rate. In the following proposition, we will use Lemma \ref{lem:finer} to bootstrap the decay rate from lemma \ref{lem:crude} to a sharper exponential decay rate.

\begin{proposition}
	\label{prop:decay-v}
	If $\nu < 1$, then there exists a constant $C > 0$ depending on $\mu$ such that
	\begin{align*}
		y _2 (t) + \left\lvert v (t) - \expo ^\nu \right\rvert + |v' (t)| \le C \expo ^{-\nu t}.
	\end{align*}
	If $\nu \ge 1$, then there exist constants $C, K > 0$ depending on $\mu$ such that
	\begin{align*}
		y _2 (t) + \left\lvert v (t) - \expo ^\nu \right\rvert + |v' (t)| \le C \langle t \rangle ^K \expo ^{-t},
	\end{align*}
	where $\langle t \rangle \coloneqq \sqrt{1+t^2}$ denotes the usual Japanese bracket shorthand.
\end{proposition}

\begin{proof}
	We first claim that for any $\alpha \le \nu$ with $\alpha < 1$, it holds that
	\begin{align*}
		\left\lvert v (s) - \expo ^\nu \right\rvert \le C _\alpha \expo ^{-\alpha t}.
	\end{align*}
	Indeed, given any $\tilde c < \min \{c', \bar c\}$, we have
	\begin{align*}
		y _{\tilde c} (0) = \int _0 ^\infty \left\lvert v (s) - \expo ^\nu \right\rvert \expo ^{\tilde c s} \dd s &= \int _0 ^\infty (- y _c ' (s) - c y _c (s)) \expo ^{\tilde c s} \dd s \\
		&= y _c (0) + \int _0 ^\infty (\tilde c - c) y _c (s) \expo ^{\tilde c s} \dd s \\
		&\le y _c (0) + Cy _c (0) \frac{\tilde c - c}{c' - \tilde c} + \begin{cases}
		\frac {C (\tilde c - c)}{(\bar c - c) (\bar c - c') (c' - \tilde c)},  & c' < \bar c \\
		\frac {C (\tilde c - c) }{(\bar c - c) (\bar c - \tilde c) (c' - \tilde c) },  & c' = \bar c \\
		\frac {C (\tilde c - c)}{(\bar c - c) (c' - \bar c) (\bar c - \tilde c)}, & c' > \bar c,
	\end{cases}
	\end{align*}
	which is convergent. We now iterate this using Lemma \ref{lem:finer}. Clearly, $y _0 (0) < \infty$ in view of Lemma \ref{lem:crude}. After finitely many steps, $c' > \alpha$, and we have
	\begin{align*}
		y _{c'} (t) + y _2 (t) \le C \expo ^{-\alpha t}.
	\end{align*}

	From this claim, we proved $y _2 (t) \le C \expo ^{-\nu t}$ for $\nu < 1$. For the case $\nu \ge 1$, we know that $y _c (0) < \infty$ for all $c < 1$. To use Lemma \ref{lem:finer}, we need to quantify the size of $y _c (0)$, which is actually the Laplace transform of $|v - \expo ^{\nu}|$ evaluated at $-c$. Although it can be estimated from the above iteration, we use the following strategy instead. Taking the derivative of $y_c(0)$ with respect to $c$ yields
	\begin{align*}
		\frac{\dd}{\dd c} y _c (0) &= \frac{\dd}{\dd c} \int _0 ^\infty \left\lvert v (t) - \expo ^\nu \right\rvert \expo ^{c t} \dd t  = \int _0 ^\infty \left\lvert v (t) - \expo ^\nu \right\rvert t \expo ^{c t} \dd t \\
		&= \int _0 ^\infty \left(-y _c' (t) - c y _c (t)\right) t \expo ^{c t} \dd t = \int _0 ^\infty y _c (t) (1 + c t - c t) \expo ^{c t} \dd t \\
		&= \int _0 ^\infty y _c (t) \expo ^{c t} \dd t.
	\end{align*}
	Using Lemma \ref{lem:finer} and bearing in mind that $\nu \ge 1$ implies that $\bar{c} = 1 > c'$, we can control it by
	\begin{align*}
		\frac{\dd}{\dd c} y _c (0) &\le \frac C{c' - c} y _c (0) + \frac{C}{(1 - c) (1 - c') (c' - c)} \\
		&\le \frac {K}{1 - c} y _c (0) + \frac {C}{(1 - c) ^3},
	\end{align*}
	whence $y _c (0) \le C (1 - c) ^{-K}$ by taking $K > 2$. Thus we also have
	\begin{align*}
		y _2 (t) \le C y _c (0) \expo ^{-c' t} + \frac C{(\bar c - c) (\bar c - c')} \expo ^{-c' t } \le C (1 - c) ^{-K} \expo ^{-c t}.
	\end{align*}
	This is true for all $c < 1$, so for $t > 1$ we can take $c = 1 - \frac1t$ and obtain
	\begin{align*}
		y _2 (t) \le C t ^K \expo ^{-t}.
	\end{align*}
	In summary, we find $y _2$ has the desired decay rate.

	To get the desired estimate for the convergence rate of $v (t)$, we employ \eqref{eq:diff-v-2}:
	\begin{align*}
		\left\lvert v (t) - \expo ^\nu \right\rvert \le \left\lvert e _2 (t) - \expo ^{-t} \varphi (\expo ^{\nu \expo ^{-t}})\right\rvert + \mu y _2 (t)
		& \le \begin{cases}
			C \expo ^{-\nu t}, & \nu < 1, \\
			C \langle t \rangle ^K \expo ^{-t}, & \nu \ge 1.
		\end{cases}.
	\end{align*}
	For the derivative estimate, we differentiate $v$ using \eqref{eq:v-error} and get
	\begin{align*}
		v' (t) &= e _2 ' (t) + \mu v (t) - \mu \int _0 ^t [v (s)] ^{\expo ^{-t + s}} \expo ^{-t + s} \dd s \\
		&\qquad - \mu \int _0 ^t [v (s)] ^{\expo ^{-t + s}} \log (v (s)) \expo ^{-2t + 2s} \dd s \\
		&= e _2 ' (t) + (\mu - 1) v (t) + \mu e _2 (t) - \mu \int _0 ^t [v (s)] ^{\expo ^{-t + s}} \log (v (s)) \expo ^{-2t + 2s} \dd s,
	\end{align*}
	in which
	\begin{align*}
		e _2' (t) = -f _0' (t) - f _0 (0) \exp \left(
			-\int _0 ^t \log v (s) \dd s
		\right) \log v (t) = \cO (\expo ^{-(1 \wedge \nu) t}).
	\end{align*}
	Note that $x \mapsto x ^{\alpha} \log x$ is Lipschitz on $[1, \expo ^\mu]$ uniformly for $\alpha \in [0, 1]$, hence
	\begin{align*}
		& \left\lvert
			\mu \int _0 ^t [v (s)] ^{\expo ^{-t + s}} \log (v (s)) \expo ^{-2t + 2s} \dd s - \mu \int _0 ^t \nu \expo ^{\nu \expo ^{-t + s}} \expo ^{-2t + 2s} \dd s
		\right\rvert \\
		& \qquad \le C \mu \int _0 ^t \left\lvert v (s) - \expo ^\nu \right\rvert \expo ^{-2t + 2s} \dd s = C y _2 (t).
	\end{align*}
	Therefore,
	\begin{align*}
		|v' (t)| &\le |e _2' (t)| + (\mu - 1) |v (t) - \expo ^\nu| + \mu e _2 (t) + C y _2 (t) \\
		&\qquad + (\mu - 1) \expo ^\nu - \mu \nu \int _0 ^t \expo ^{\nu \expo ^{-s}} \expo ^{-2 s} \dd s \\
		&\le |e _2' (t)| + (\mu - 1) |v (t) - \expo ^\nu| + \mu e _2 (t) + C y _2 (t) + \mu \nu \int _t ^\infty \expo ^{\nu \expo ^{-s}} \expo ^{-2 s} \dd s \\
		& \le \begin{cases}
			C \expo ^{-\nu t}, & \nu < 1, \\
			C \langle t \rangle ^K \expo ^{-t}, & \nu \ge 1.
		\end{cases}
	\end{align*}
	This completes the proof.
\end{proof}

As a corollary, we deduce the following convergence rate of $p _1 (t) \to \mu - \nu$ as $t \to \infty$.

\begin{corollary}
	\label{cor:p1-convergence}
	For $t \ge 0$, we have
	\begin{align*}
		|p _1 (t) - \mu + \nu| \le \begin{cases}
			C \expo ^{-\nu t}, & \nu < 1, \\
			C \langle t \rangle^K \expo ^{-t}, & \nu \ge 1.
		\end{cases}.
	\end{align*}
\end{corollary}

\begin{proof}
	It suffices to notice that $\mu - p _1 (t)$ can be recovered from $v$ via
	\begin{align*}
		\mu - p _1 (t) = \frac{v' (t)}{v (t)} + \log v (t).
	\end{align*}
	Since both $v' (t) \xrightarrow{t \to \infty} 0$ and $v (t) \xrightarrow{t \to \infty} \expo ^\nu$ occur at this rate, the result follows.
\end{proof}

\subsection{Strong convergence of the ODE system}

In this subsection, we are ready to prove the various strong convergence results regarding the solution $\bp (t)$ of the system \eqref{eq:law_limit_repeat} towards the zero-truncated Poisson distribution $\bpb$ \eqref{eq:zero_truncated_poisson} as $t \to +\infty$. First, we show convergence in $\ell ^2$. Recall that
\begin{align*}
	\left\lVert \bp (t) - \bpb \right\rVert _{\ell ^2} ^2 = \sum _{n = 0} ^\infty \left( p _n (t) - \pb _n \right) ^2.
\end{align*}

\begin{theorem}\label{thm:2}
	Let $\mu > 1$. There exists constants $C$, $K$ depending on $\mu$ such that for all $t \ge 0$, it holds that
	\begin{align*}
		\left\lVert \bp (t) - \bpb \right\rVert _{\ell ^2} \le \begin{cases}
			C \expo ^{-\nu t}, & \nu < 1, \\
			C \langle t \rangle ^K \expo ^{-t}, & \nu \ge 1.
		\end{cases}.
	\end{align*}
\end{theorem}

\begin{proof}
	We first recall the classical Parseval's identity:
	\begin{align*}
		\left\lVert \bp (t) - \bpb \right\rVert _{\ell ^2} ^2 = \frac1{2 \pi} \int _0 ^{2 \pi} \left\lvert
			\G (t, \expo ^{i \theta}) - \G _{\bpb} (\expo ^{i \theta})
		\right\rvert ^2 \dd \theta.
	\end{align*}
	By Lemma \ref{lem:solution-to-pde}, we have
	\begin{align*}
		(\G (t, 1 - z) - 1) [v (t)] ^z
		&= \G (0, 1 - z \expo ^{-t}) - 1 - \mu z \int _0 ^t [v (s)] ^{z \expo ^{-t + s}} \expo ^{-t + s} \dd s
	\end{align*}
	for all $z \in \mathbb C$ with $|z - 1| \le 1$.  Notice that
	\begin{align*}
		\lvert \G (0, 1 - z \expo ^{-t}) - 1 \rvert &\le \partial _z G (0, 1) \lvert z \rvert \expo ^{-t} \le C \expo ^{-t}, \\
		\left|
			\int _0 ^t [v (s)] ^{z \expo ^{-t + s}} \expo ^{-t + s} \dd s - \int _0 ^t \expo ^{\nu z \expo ^{-t + s}} \expo ^{-t + s} \dd s
		\right| &\le C |z| \int _0 ^t \lvert v (s) - \expo ^\nu \rvert \expo ^{-2t + 2s} \dd s \\
		&\le C y _2 (t), \\
		\left|
			\mu \int _0 ^t \expo ^{\nu z \expo ^{-t + s}} \expo ^{-t + s} \dd s - \varphi (\expo ^{\nu z})
		\right| &\le \expo ^{-t} \left|\varphi (\expo ^{\nu z \expo ^{-t}})\right| \le C \expo ^{-t}.
	\end{align*}
	On the other hand, we know for $z \in \mathbb C$ with $|z - 1| \le 1$ that
	\begin{align*}
		\left|
			(\G (t, 1 - z) - 1) [v (t)] ^z - (\G (t, 1 - z) - 1) \expo ^{\nu z}
		\right| \le C |v (t) - \expo ^\nu|.
	\end{align*}
	Assembling these estimates, we proved for $z \in \mathbb C$ with $|z - 1| \le 1$ that
	\begin{align*}
		\left|
			(G (t, 1 - z) - 1) \expo ^{\nu z} + z \varphi (\expo ^{\nu z})
		\right| \le \begin{cases}
			C \expo ^{-\nu t}, & \nu < 1, \\
			C \langle t \rangle ^K \expo ^{-t}, & \nu \ge 1.
		\end{cases}
	\end{align*}
	Since
	\begin{align*}
		G _{\bpb} (1 - z) = \frac1{\expo ^\nu - 1}\sum _{n = 1} ^\infty \frac{\nu ^n}{n!} (1 - z) ^n = \frac{\expo ^{\nu (1 - z)} - 1}{\expo ^\nu - 1} =
		 1 - \expo ^{-\nu z} z \varphi (\expo ^{\nu z}),
	\end{align*}
	the above implies uniform convergence of $\G (t, 1 - z)$ to $\G _{\bpb} (1 - z)$ for all $z \in \mathbb C$ with $|z - 1| \le 1$, which shows that $\bp (t)$ converges to $\bpb$ in $\ell ^2$ by Parseval's identity.
\end{proof}

Utilizing the tail estimate for the (zero-truncated) Poisson distribution, we can also establish the following convergence result in $\ell^1$.

\begin{corollary}
	\label{cor:2}
	Let $\mu > 1$. There exists constants $C$, $K$ depending on $\mu$ such that for all $t \ge 0$, it holds that
	\begin{align*}
		\left\lVert \bp (t) - \bpb \right\rVert _{\ell^1} \le \begin{cases}
			C \langle t \rangle ^\frac12 \expo ^{-\nu t}, & \nu < 1, \\
			C \langle t \rangle ^{K + \frac12} \expo ^{-t}, & \nu \ge 1.
		\end{cases}.
	\end{align*}
\end{corollary}

\begin{proof}
	For $x \in \mathbb N$ to be specified later, we have
	\begin{align*}
		\left\lVert \bp (t) - \bpb \right\rVert _{\ell ^1} &\le \left\lVert \bp (t) - \bpb \right\rVert _{\ell ^1 ([0, x])} + \left\lVert \bp (t) - \bpb \right\rVert _{\ell ^1 ([x + 1, \infty))} \\
		&\le 2 \left\lVert \bp (t) - \bpb \right\rVert _{\ell ^1 ([0, x])} + 2 \left\lVert \bpb \right\rVert _{\ell ^1 ([x + 1, \infty))}
	\end{align*}
	The first term is easily controlled by the $\ell ^2$ norm:
	\begin{align*}
		\left\lVert \bp (t) - \bpb \right\rVert _{\ell ^1 ([0, x])} \le \sqrt x \left\lVert \bp (t) - \bpb \right\rVert _{\ell ^2}.
	\end{align*}
	The second term is amenable to explicit computations, leading us to
	\begin{align*}
		\left\lVert \bpb \right\rVert _{\ell ^1 ([x + 1, \infty))} = \sum _{n = x + 1} ^\infty \pb _n = \frac1{\expo ^\nu - 1} \sum _{n = x + 1} ^\infty \frac{\nu ^n}{n!} .
	\end{align*}
	Thanks to the Chernoff bound for the Poisson distribution, for $x \ge \nu$ it holds that
	\begin{align*}
		\sum _{n = x} ^\infty \frac{\nu ^n \expo ^{-\nu}}{n!} \le \frac{(\expo \nu) ^x \expo ^{-\nu}}{x ^x}.
	\end{align*}
	We know for our zero-truncated Poisson distribution that
	\begin{align*}
		\left\lVert \bpb \right\rVert _{\ell ^1 ([x + 1, \infty))} \le \frac1{\expo ^\nu - 1} \left(\frac{\expo \nu}{x}\right) ^x.
	\end{align*}
	Finally, setting $x = [t \vee \nu \expo ^2]$ allows us to deduce that $\left\lVert \bpb \right\rVert _{\ell ^1 ([x + 1, \infty))} \le C \expo ^{-t}$, whence
	\begin{align*}
		\left\lVert \bp (t) - \bpb \right\rVert _{\ell ^1} \le \begin{cases}
			C \langle t \rangle ^\frac12 \expo ^{-\nu t}, & \nu < 1 \\
			C \langle t \rangle ^{K + \frac12} \expo ^{-t}, & \nu \ge 1.
		\end{cases}.
	\end{align*}
	This completes the proof.
\end{proof}

\begin{figure}[htbp]
	\centering
	\includegraphics[width=.75\textwidth]{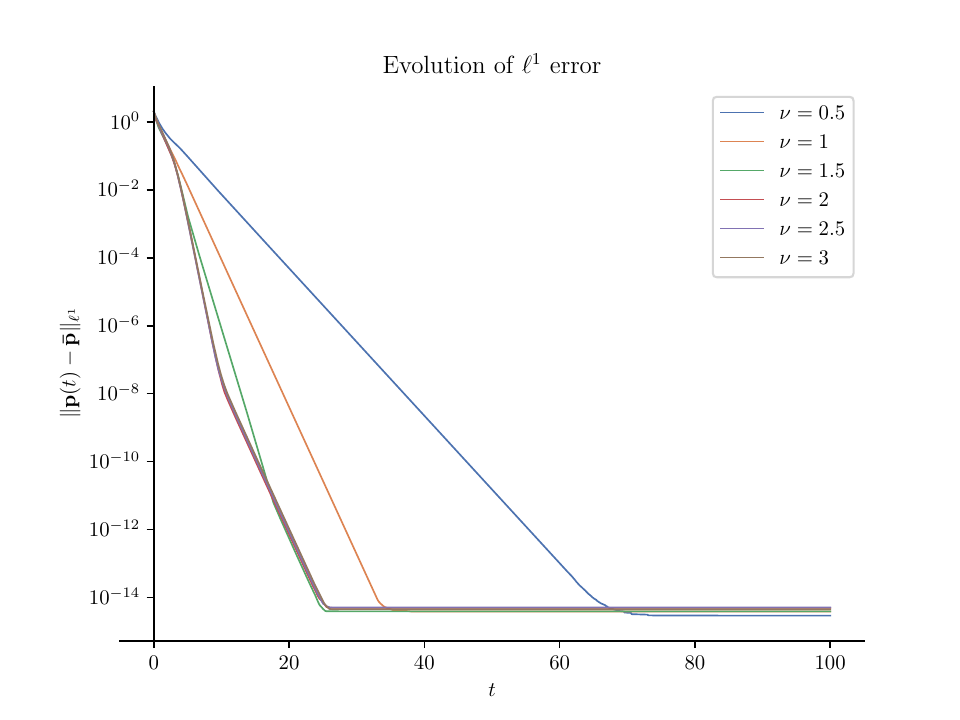}
	\caption{Evolution of the $\ell ^1$ error $\left\lVert \bp (t) - \bpb \right\rVert _{\ell ^1}$ over time for different values of $\nu$. It can be seen that larger values of $\nu$ (or $\mu$) leads to faster convergence, although such improvement in terms of the convergence rate saturates when $\nu$ becomes large enough.}
	\label{fig:ODE_ell2}
\end{figure}

To illustrate the quantitative convergence guarantee reported in Corollary \ref{cor:2} we plot the evolution of the $\ell^1$ error $\|\bp(t) - \bpb\|_{\ell^1}$ over time (see Figure \ref{fig:ODE_ell2}) with $\nu \in \{0.5, 1, 1.5, 2, 2.5, 3\}$, under the same set-up as used for Figure \ref{fig:ODEsimulation}. We observe the exponential decay of $\|\bp(t) - \bpb\|_{\ell^1}$ as predicted by Theorem \ref{thm:2}, although our analytical rate might be sub-optimal for $\nu \ge 1$. Meanwhile, we also emphasize that Figure \ref{fig:ODE_ell2} contains a part where the machine precision is reached and hence the numerical error can no longer decay further. Lastly, we observe that the numerical experiments displayed above for ``large'' $\nu$ seem to demonstrate two regimes in the convergence behavior: $\expo ^{-(2 \wedge \nu) t}$ in the first period, and $\expo ^{-(1 \wedge \nu) t}$ afterwards.

To understand the reason for the two-scale behavior, we examine the convergence for $\nu > 1$ in more detail, and provide a non-rigorous explanation here. Instead of convergence rate of $v$, we examine the convergence rate of $a$ directly. We do so by linearize the Volterra-type integral equation. Let $u (t) = \log v (t) - \nu$, and denote the negative exponential function as $\expo ^-{} (t) = \expo ^{-t}$, then \eqref{eq:auxillary} can be rewritten as one-sided convolution: 
\begin{align}
	\label{eqn:convolution-1}
	u = \expo ^-{} * a - \nu = \expo ^-{} * (a - \nu) - \nu \expo ^-{}.
\end{align}
Moreover, since $v = \expo ^\nu \expo ^u$, we rewrite \eqref{eq:diff-v-0} with $u$ as 
\begin{align*}
	\expo ^\nu (\expo ^{u (t)} - 1) &= e _2 (t) - \expo ^{-t} \varphi (\expo ^{\nu \expo ^{-t}}) + \mu \int _0 ^t \left(
		[v (s)] ^{\expo ^{-t + s}} - \expo ^{\nu \expo ^{-t + s}}
	\right) \expo ^{-t + s} \dd s \\
	&= e _2 (t) - \expo ^{-t} \varphi (\expo ^{\nu \expo ^{-t}}) + \mu \int _0 ^t \left(
		\expo ^{u (s) \expo ^{-t + s}} - 1
	\right) \expo ^{\nu \expo ^{-t + s}} \expo ^{-t + s} \dd s.
\end{align*}
Denote $k (t) = \mu \expo ^{\nu \expo ^{-t}} \expo ^{-2 t}$. We can write the above equation in convolution form:
\begin{align}
	\label{eqn:convolution-2}
	\expo ^\nu u = e _3 + k * u
\end{align}
with a remainder 
\begin{align*}
	e _3 (t) &:= e _2 (t) - \expo ^{-t} \varphi (\expo ^{\nu \expo ^{-t}}) - \expo ^\nu (\expo ^{u (t)} - 1 - u (t)) \\
	& \qquad + \mu \int _0 ^t \left(
		\expo ^{u (s) \expo ^{-t + s}} - 1 - u (s) \expo ^{-t + s}
	\right) \expo ^{\nu \expo ^{-t + s}} \expo ^{-t + s} \dd s.
\end{align*}
From Proposition \ref{prop:decay-v} we know $|u (t)| = \cO (\langle t \rangle ^K \expo ^{-t})$. Ignore the algebraic factor $\langle t \rangle ^K$ for now, and use $\varphi (\expo ^{\nu \expo ^{-t}}) = \varphi (1) + \cO (\expo ^{\nu \expo ^{-t}} - 1) = \mu + \mathcal O (\expo ^{-t})$, we know
\begin{align*}
	e _3 (t) = e _2 (t) - \mu \expo ^{-t} + \cO (\expo ^{-2t}).
\end{align*}
Combine \eqref{eqn:convolution-1} and \eqref{eqn:convolution-2}, we have
\begin{align}
	\notag
	\expo ^\nu \expo ^-{} * (a - \nu) - \nu \expo ^\nu \expo ^-{} &= e _3 + k * \expo ^-{} * (a - \nu) - k * \nu \expo ^-{} \\
	\label{eqn:convolution-3}
	\implies (\expo ^\nu \expo ^-{} - k * \expo ^-{}) * (a - \nu) &= e _3 + \nu \expo ^\nu \expo ^-{} - k * \nu \expo ^-{}.
\end{align}
We can compute one of the convolution explicitly:
\begin{align*}
	k * \expo ^-{} (t) &= \mu \int _0 ^t \expo ^{\nu \expo ^{-s}} \expo ^{-2 s} \expo ^{-t + s} \dd s \\
	&= \expo ^{-t} \left(
		\mu \int _0 ^\infty \expo ^{\nu \expo ^{-s}} \expo ^{-s} \dd s - \mu \int _t ^\infty \expo ^{\nu \expo ^{-s}} \expo ^{-s} \dd s
	\right) \\
	&= \expo ^{-t} \left(\expo ^\nu - \expo ^{-t} \varphi (\expo ^{\nu \expo ^{-t}})\right).
\end{align*}
Thus \eqref{eqn:convolution-3} becomes
\begin{align}
	\notag
	(\expo ^-{}) ^2 \varphi (\expo ^{\nu \expo ^-{}}) * (a - \nu) &= e _3 + \nu (\expo ^-{}) ^2 \varphi (\expo ^{\nu \expo ^-{}}) \\
	\label{eqn:convolution-4}
	&= e _2 - \mu \expo ^{-t} + \cO ((\expo ^{-}) ^2).
\end{align}
Recall that 
\begin{align*}
	e _2 (t) = 1 - \G (0, 1 - \expo ^{-t}) + \cO (\expo ^{-\nu t}),
\end{align*}
where 
\begin{align}
	\label{eqn:G-estimate}
	\G (0, 1 - \expo ^{-t}) &= \G (0, 1) - \expo ^{-t} \partial _z G (0, 1) + \cO (\expo ^{-2t}) = 1 - \mu \expo ^{-t} + \cO (\expo ^{-2t})
\end{align}
assuming finite second moment. 
Thus, we have a cancellation on the right-hand side of \eqref{eqn:convolution-4} and we are left with 
\begin{align}
	\label{eqn:convolution-5}
	(\expo ^-{}) ^2 \varphi (\expo ^{\nu \expo ^-{}}) * (a - \nu) (t) = \cO (\expo ^{-\nu t}) + \cO (\expo ^{-2t}).
\end{align}
The convolution kernel has decay rate $\expo ^{-2t} \varphi (\expo ^{\nu \expo ^{-t}}) = \cO (\expo ^{-2t})$. Therefore, provided that $a - \nu$ has an asymptotic leading term $C \expo ^{-\alpha t}$ for some $\alpha$, then $\alpha = 2 \wedge \nu$ (unless there exists a further cancellation). We thus conjure that the sharp decay rate should be $\expo ^{-(2 \wedge \nu) t}$ for $\mu > 1$, which is supported by the experiments in Figure \ref{fig:ODE_ell2}. 

Supopse that, due to simulation machine error, the first moment of the initial condition $\partial _z G (0, 1)$ is not $\mu$ but slightly off, then in \eqref{eqn:G-estimate} the $\expo ^{-t}$ term cannot perfectly cancel $\mu \expo ^{-t}$ in \eqref{eqn:convolution-4}, and the right-hand side of \eqref{eqn:convolution-5} would have a decay rate of $\expo ^{-t}$ instead of $\expo ^{-(2 \wedge \nu) t}$. We believe this is the reason why we observe a change in decay rate for $\nu > 1$ in Figure \ref{fig:ODE_ell2} beneath a certain scale.

\section{Conclusion}
\label{sec:sec5}

In this work, we studied the Fokker--Planck equation associated with the mean-field limit of the so-called dispersion process (on a complete graph with $N$ vertices) introduced and studied in a number of recent works \cite{cooper_dispersion_2018,de_dispersion_2023,frieze_note_2018,shang_longest_2020}. We prove quantitative global convergence results for the solution of the Fokker--Planck equation towards the equilibrium. The equilibrium and the obtained convergence rate differ in the underpopulated regime ($\mu < 1$) and the overpopulated regime ($\mu > 1$). 

This work also leaves some important follow-up problems which deserve separate treatments and attentions on their own. For instance, can we design a natural Lyapunov functional associated to the solution of the infinite dimensional nonlinear ODE system \eqref{eq:law_limit} when $\mu > 1$? Is it possible to improve the various decay rates reported in Theorem \ref{thm:main}? We believe that answers to these questions will enhance our understanding on the role played by the model parameter $\mu$ and contribute to the ever-growling literature on the large time analysis of mean-field dynamical systems arising from social-economic, life, and natural sciences \cite{cao_fractal_2024,cao_iterative_2024,jabin_review_2014,pareschi_interacting_2013}. 

\begin{appendix}

\section{Uniqueness of solution to the mean-field ODE system}
\label{app:uniqueness}
\newcommand{\bpi}{\boldsymbol \pi}
\newcommand{\bga}{\boldsymbol \gamma}

We provide a proof for the uniqueness of solution to the mean-field dynamical system \eqref{eq:law_limit} under the framework of Wasserstein distances. The proof is inspired by the coupling method (see \cite[section 7.5]{villani_topics_2003}). To illustrate the main ideas, let us recast the equation into the most general form. Let $\bp$, $\tilde \bp: [0, T] \to \mathbb R ^\mathbb N$ be two solutions to the linear Fokker--Planck equations with different drifts
\begin{align}
	\label{eqn:fokker-planck-1}
	\bp' (t) &= \mathcal D _+ [\mathbf b (t) \bp (t)] - \mathcal D _- [\mathbf a (t) \bp (t)], \\ 
	\label{eqn:fokker-planck-2}
	\tilde \bp ' (t) &= \mathcal D _+ [\tilde {\mathbf b} (t) \tilde \bp (t)] - \mathcal D _- [\tilde {\mathbf a} (t) \tilde \bp (t)].
\end{align}
where $\mathbf a, \mathbf b, \tilde{\mathbf a}, \tilde{\mathbf b}: [0, T] \to \mathbb R ^{\mathbb N}$ are upward/downward drift, and the product is in the pointwise sense. A transport plan $\bpi = \{\pi _{m, n}\} _{m, n = 0} ^\infty$ between $\bq$ and $\tilde \bq$ is a probability measure on $\mathbb N \times \mathbb N$ satisfying 
\begin{align*}
	\pi _{m, n} &\ge 0, &
	\sum _{n = 0} ^\infty \pi _{m, n} &= q _m, & 
	\sum _{m = 0} ^\infty \pi _{m, n} &= \tilde q _n, & 
	\sum _{m, n = 0} ^\infty \pi _{m, n} &= 1.
\end{align*}
The last condition is redundant because $\bq$ and $\tilde \bq$ are both probability measures.
The set of transport plans between $\bq$ and $\tilde \bq$ is denoted as $\Pi (\bq, \tilde \bq)$. The 1-Wasserstein distance is defined as 
\begin{align*}
	\mathcal W (\bq, \tilde \bq) = \inf _{\bpi \in \Pi (\bq, \tilde \bq)} \sum _{m, n = 0} ^\infty \lvert m - n \rvert \pi _{m, n}.
\end{align*}

Let $\bga _0 \in \Pi (\bp (0), \tilde \bp (0))$ be a transport plan between the initial conditions. We want to design a transport plan $\bga (t) \in \Pi (\bp (t), \tilde \bp (t))$. To do this, we first need to generalize the forward/backward difference operator to $\mathbb N ^2$:
\begin{align*}
	\mathcal D _{++} [\bpi] _{m, n} &= \pi _{m + 1, n + 1} - \pi _{m, n}, &
	\mathcal D _{--} [\bpi] _{m, n} &= \pi _{m, n} - \pi _{m - 1, n - 1}, \\
	\mathcal D _{+\circ} [\bpi] _{m, n} &= \pi _{m + 1, n} - \pi _{m, n}, &
	\mathcal D _{-\circ} [\bpi] _{m, n} &= \pi _{m, n} - \pi _{m - 1, n}, \\
	\mathcal D _{\circ+} [\bpi] _{m, n} &= \pi _{m, n + 1} - \pi _{m, n}, &
	\mathcal D _{\circ-} [\bpi] _{m, n} &= \pi _{m, n} - \pi _{m, n - 1}.
\end{align*}
We define the drifts $\mathbf A ^{--}, \mathbf A ^{-\circ}, \mathbf A ^{\circ-}$ and $\mathbf B ^{++}, \mathbf B ^{+\circ}, \mathbf B ^{\circ+}: [0, T] \to \mathbb R ^{\mathbb N \times \mathbb N}$ by
\begin{align*}
	A ^{--} _{m, n} (t) &= a _m (t) \wedge a _n (t), &
	B ^{++} _{m, n} (t) &= b _m (t) \wedge b _n(t), \\
	A ^{-\circ} _{m, n} (t) &= [a _m (t) - a _n (t)] _+, &
	B ^{+\circ} _{m, n} (t) &= [b _m (t) - b _n(t)] _+, \\
	A ^{\circ-} _{m, n} (t) &= [a _n (t) - a _m (t)] _+, &
	B ^{\circ+} _{m, n} (t) &= [b _n (t) - b _m (t)] _+.
\end{align*}
Here $\mathbf a = \{a _m\} _{m = 0} ^\infty$, $\mathbf b = \{b _m\} _{m = 0} ^\infty$, $\mathbf A ^{--} = \{A ^{--} _{m, n}\} _{m, n = 0} ^\infty$, etc., and $(x) _+ = \max\{x, 0\}$ denotes the positive part.
Define $\bga: [0, T] \to \mathbb R ^{\mathbb N \times \mathbb N}$ to be the solution to the following Fokker--Planck equation in $\mathbb N ^2$ with initial condition $\bga (0) = \bga _0$:
\begin{align*}
	\bga' (t) &= \mathcal D _{++} [\mathbf B ^{++} (t) \bga (t)] + \mathcal D _{+\circ} [\mathbf B ^{+\circ} (t) \bga (t)] + \mathcal D _{\circ+} [\mathbf B ^{\circ+} (t) \bga (t)] \\
	&\qquad - \mathcal D _{--} [\mathbf A ^{--} (t) \bga (t)] - \mathcal D _{-\circ} [\mathbf A ^{-\circ} (t) \bga (t)] - \mathcal D _{\circ-} [\mathbf A ^{\circ-} (t) \bga (t)].
\end{align*}
Same as \eqref{eqn:fokker-planck-1} and \eqref{eqn:fokker-planck-2}, the drifts are passive rather than active, so it is indeed a well-posed linear Fokker--Planck equation. To see that the solution is a transport plan, we first examine the forward drift of the first marginal. From the definition of $\mathcal D _{\circ -}$ and $\mathcal D _{- \circ}$, we know that for any $\mathbf C \in \mathbb R ^{\mathbb N \times \mathbb N}$ the following identities hold:
\begin{align*}
	\sum _{n = 0} ^\infty \mathcal D _{\circ-} [\mathbf C] _{m, n} &= \sum _{n = 0} ^\infty C _{m, n} - C _{m, n - 1} = 0, \\
	\sum _{n = 0} ^\infty \mathcal D _{-\circ} [\mathbf C ] _{m, n} &= \sum _{n = 0} ^\infty C _{m, n} - C _{m - 1, n} = \sum _{n = 0} ^\infty C _{m, n} - C _{m - 1, n - 1} = \sum _{n = 0} ^\infty \mathcal D _{--} [\mathbf C] _{m, n}.
\end{align*}
Therefore,  
\begin{align*}
	& \sum _{n = 0} ^\infty \mathcal D _{--} [\mathbf A ^{--} (t) \bpi] _{m, n} + \mathcal D _{-\circ} [\mathbf A ^{-\circ} (t) \bpi] _{m, n} + \mathcal D _{\circ-} [\mathbf A ^{\circ-} (t) \bga (t)] _{m, n} \\
	& \qquad = \sum _{n = 0} ^\infty \mathcal D _{--} [\mathbf A ^{--} (t) \bpi] _{m, n} + \mathcal D _{--} [\mathbf A ^{-\circ} (t) \bpi] _{m, n} + 0\\
	& \qquad = \sum _{n = 0} ^\infty \mathcal D _{--} [(\mathbf A ^{--} (t) + \mathbf A ^{-\circ} (t))\bpi] _{m, n} \\ 
	& \qquad = \sum _{n = 0} ^\infty a _m (t) \pi _{m, n} - a _{m - 1} (t) \pi _{m - 1, n - 1} \\
	&\qquad = \mathcal D _- [\mathbf a (t) \bpi _1] _m.
\end{align*}
Here $\bpi _1$ is the first marginal of $\bpi$. By a similar computation to the backward drift, we conclude the first marginal $\bga _1$ of $\bga$ solves a linear Fokker--Planck equation 
\begin{align*}
	\bga' _1 (t) &= \mathcal D _+ [\mathbf b (t) \bga _1 (t)] - \mathcal D _- [\mathbf a (t) \bga _1 (t)].
\end{align*}
From \eqref{eqn:fokker-planck-1} we know $\bp$ and $\bga _1$ solve the same linear equation \eqref{eqn:fokker-planck-1}. They share then same initial condition, so by the uniqueness of linear Fokker--Planck equation we know $\bga _1 (t) = \bp (t)$ for all $t \in [0, T]$. Similarly the second marginal $\bga _2 (t) = \tilde \bp (t)$, and hence $\bga (t) \in \Pi (\bp (t), \tilde \bp (t))$. 

With this transport plan, we can estimate the Wasserstein distance
\begin{align*}
	\mathcal W (\bp (t), \tilde \bp (t)) \le D (t) \coloneqq \sum _{m, n = 0} ^\infty |m - n| \gamma _{m, n} (t).
\end{align*}
Take derivative, we know 
\begin{align*}
	D' (t) = \sum _{m, n = 0} ^\infty |m - n| \bigg(
		&\mathcal D _{++} [\mathbf B ^{++} (t) \bga (t)] + \mathcal D _{+\circ} [\mathbf B ^{+\circ} (t) \bga (t)] + \mathcal D _{\circ+} [\mathbf B ^{\circ+} (t) \bga (t)] \\
		- &\mathcal D _{--} [\mathbf A ^{--} (t) \bga (t)] - \mathcal D _{-\circ} [\mathbf A ^{-\circ} (t) \bga (t)] - \mathcal D _{\circ-} [\mathbf A ^{\circ-} (t) \bga (t)]
	\bigg).
\end{align*}
From the definition of $\mathcal D _{++}$ and $\mathcal D _{--}$, immediately we have
\begin{align*}
	\sum _{m, n = 0} ^\infty |m - n| 
		\mathcal D _{++} [\mathbf B ^{++} (t) \bga (t)] = \sum _{m, n = 0} ^\infty |m - n|  \mathcal D _{--} [\mathbf A ^{--} (t) \bga (t)] = 0.
\end{align*}
Moreover, using the discrete integration by parts, we have
\begin{align*}
	\sum _{m, n = 0} ^\infty |m - n| \mathcal D _{-\circ} [\mathbf A ^{-\circ} (t) \bga (t)] _{m, n} &= \sum _{m, n = 0} ^\infty -\mathcal D _{+\circ} [|m - n|] A ^{-\circ} _{m, n} (t) \gamma _{m, n} (t) \\
	&\le \sum _{m, n = 0} ^\infty A ^{-\circ} _{m, n} (t) \gamma _{m, n} (t).
\end{align*}
Similar computations yield
\begin{align}
	\notag
	D' (t) &\le \sum _{m, n = 0} ^\infty \left(
		A ^{-\circ} _{m, n} (t) + A ^{\circ-} _{m, n} (t) + B ^{+\circ} _{m, n} (t) + B ^{\circ+} _{m, n} (t)
	\right) \gamma _{m, n} (t) \\
	\label{eq:diff-ineq-2}
	&= \sum _{m, n = 0} ^\infty \left(
		|a _m (t) - a _n (t)| + 
		|b _m (t) - b _n (t)|
	\right) \gamma _{m, n} (t).
\end{align}
If the right-hand side can be controlled by $D (t)$ then we would have a differential inequality.

Now, let $\bp$ and $\tilde \bp$ be two solutions to \eqref{eq:law_limit}. Then they solve \eqref{eqn:fokker-planck-1} and \eqref{eqn:fokker-planck-2} with the drifts
\begin{align*}
	a _m (t) &= \sum _{n = 2} ^\infty n p _n (t) = \mu - p _1 (t), & \tilde a _m (t) &= \mu - \tilde p _1 (t), & b _n (t) = \tilde b _n (t) &= n \mathbf1 _{\{n \ge 2\}}.
\end{align*}
The backward drift part in \eqref{eq:diff-ineq-2} can be controlled by 
\begin{align*}
	\sum _{m, n = 0} ^\infty |b _m (t) - b _n (t)| \gamma _{m, n} (t) \le \sum _{m, n = 0} ^\infty |m - n| \gamma _{m, n} (t) = D (t).
\end{align*}
As for the forward drift in \eqref{eq:diff-ineq-2}, it is independent of $m, n$, so 
\begin{align*}
	\sum _{m, n = 0} ^\infty |a _m (t) - a _n (t)| \gamma _{m, n} (t) &= \sum _{m, n = 0} ^\infty |\mu - p _1 (t) - \mu + \tilde p _1 (t)| \gamma _{m, n} (t) \\
	&= |p _1 (t) - \tilde p _1 (t)|.
\end{align*}
It remains to show that $|p _1 (t) - \tilde p _1 (t)| \le D (t)$, which is elementary:
\begin{align*}
	D (t) \ge \sum _{m \neq n} \gamma _{m, n} (t) &\ge \sum _{m = 0} ^\infty \gamma _{m, 1} (t) + \sum _{n = 0} ^\infty \gamma _{1, n} (t) - 2 \gamma _{1, 1} (t) \\
	&= p _1 (t) + \tilde p _1 (t) - 2 \gamma _{1, 1} (t).
\end{align*}
Because $\gamma _{1, 1} (t) \le p _1 (t) \wedge \tilde p _1 (t)$, we have $|p _1 (t) - \tilde p _1 (t)| \le D (t)$.

To summarize, we have shown that 
\begin{align*}
	D' (t) \le 2 D (t). 
\end{align*}
Gronwall's inequality provides an upper bound on the Wasserstein distance between two solutions: 
\begin{align*}
	\mathcal W (\bp (t), \tilde \bp (t)) \le D (t) \le D (0) \expo ^{2 t}.
\end{align*}
Taking infimum of all transport plans at time zero, we have 
\begin{align*}
	\mathcal W (\bp (t), \tilde \bp (t)) \le \mathcal W (\bp (0), \tilde \bp (0)) \expo ^{2 t}.
\end{align*}
If $\bp (0) = \tilde \bp (0)$, then $\mathcal W (\bp (t), \tilde \bp (t)) = 0$, which implies $\bp (t) = \tilde \bp (t)$ for all $t > 0$.

\end{appendix}

\end{document}